\documentclass[a4paper, reqno, 11pt]{amsart}

\usepackage[english]{babel}
\usepackage{amsmath}
\usepackage{mathrsfs}
\usepackage{amssymb,latexsym}
\usepackage{amsthm}
\usepackage{enumerate}
\usepackage{ifthen}
\usepackage{bbm}
\usepackage{color}
\usepackage{xcolor}
\usepackage{graphicx}
\provideboolean{shownotes} 
\setboolean{shownotes}{true}
\usepackage{hyperref}
\usepackage{geometry}
\usepackage{esint}
\newcommand{\margnote}[1]{
\ifthenelse{\boolean{shownotes}}%
{\marginpar{\raggedright\tiny\texttt{#1}}}%
{}%
}

\newcommand{\hole}[1]{
\ifthenelse{\boolean{shownotes}}%
{ \begin{center}\fbox{ \rule {.25cm}{0cm}
\rule[-.1cm]{0cm}{.4cm} \parbox{.85\textwidth}{
\texttt{#1}} \rule {.25cm}{0cm}}\end{center}}
{}
}
\newtheorem{theorem}{Theorem}[section]
\newtheorem{proposition}[theorem]{Proposition}
\newtheorem{definition}[theorem]{Definition}
\newtheorem{lemma}[theorem]{Lemma}

\theoremstyle{remark}
\newtheorem{remark}[theorem]{Remark}
\newtheorem*{thm}{{\bf Lemma A}}

\DeclareMathOperator{\dive}{div}

\def\d{{\partial}}
\newcommand{\e}{\epsilon}	
\newcommand{\T}{\mathbb{T}^3}
\newcommand{\MT}{\mathcal{T}}
\newcommand{\rrho}{\sqrt{\rho}}
\newcommand{\R}{\mathbb{R}}

\newcommand{\wn}{w_{{\e}}}
\newcommand{\re}{\rho_{\e}}
\newcommand{\rre}{\sqrt{\rho_\e}}
\newcommand{\ue}{u_{\e}}
\newcommand{\weakto}{\rightharpoonup}
\newcommand{\weaktos}{\stackrel{*}{\rightharpoonup}}

\numberwithin{equation}{section}

\allowdisplaybreaks

\title[Relaxation from Quantum Navier--Stokes  to  Quantum Drift--Diffusion]{Relaxation limit from the Quantum Navier--Stokes equations to the Quantum Drift--Diffusion equation}

\author[P. Antonelli]{Paolo Antonelli}
\address[Paolo Antonelli]{GSSI-Gran Sasso Science Institute (Italy)}
\email{paolo.antonelli@gssi.it}

\author[G. Cianfarani Carnevale]{Giada Cianfarani Carnevale}
\address[Giada Cianfarani Carnevale]{Dipartimento di Ingegneria e Scienze dell'Informazione e Matematica, Universit\`a degli Studi dell'Aquila (Italy)}
\email{giada.cianfaranicarnevale@graduate.univaq.it}

\author[C. Lattanzio]{Corrado Lattanzio}
\address[Corrado Lattanzio]{Dipartimento di Ingegneria e Scienze dell'Informazione e Matematica, Universit\`a degli Studi dell'Aquila (Italy)}
\email{corrado@univaq.it} 

\author[S. Spirito]{Stefano Spirito}
\address[Stefano Spirito ]{Dipartimento di Ingegneria e Scienze dell'Informazione e Matematica, Universit\`a degli Studi dell'Aquila (Italy)}
\email{stefano.spirito@univaq.it}

\date\today

\begin{document}
\maketitle
\begin{abstract}
The relaxation-time limit from the Quantum-Navier-Stokes-Poisson system to the quantum drift-diffusion equation is performed in the framework of finite energy weak solutions. No assumptions on the limiting solution are made. The proof exploits the suitably scaled a priori bounds inferred by the energy and BD entropy estimates. Moreover, it is shown how from those estimates the Fisher entropy and free energy estimates associated to the diffusive evolution are recovered in the limit. As a byproduct, our main result also provides an alternative proof for the existence of finite energy weak solutions to the quantum drift-diffusion equation.
\end{abstract}
\section{Introduction}
This paper studies the relaxation-time limit for the Quantum Navier-Stokes-Poisson (QNSP) system with linear damping, towards the quantum drift-diffusion equation. More precisely, in the three dimensional torus $\mathbb T^3$, we consider a compressible, viscous fluid, whose dynamics is prescribed by
\begin{equation}\label{eq:qns_intro}
\begin{aligned} 
& \partial_t\rho + \dive(\rho u ) = 0 \\ 
& \partial_t(\rho u) + \dive (\rho u \otimes u) - \dive (\rho Du) + \nabla \rho^{\gamma}+ \rho \nabla V  = 2\rho\nabla \left( \frac{\Delta \sqrt{\rho}}{\sqrt{\rho} } \right) - \xi \rho u\\
&- \Delta V = \rho-g.
\end{aligned} 
\end{equation} 
Here the unknowns $\rho$, $u$, and $V$ denote the particle density, the velocity field, and   the electrostatic potential respectively. The function $g$ is given and represents the doping profile. \\
The system arises in the macroscopic description of electron transport in nanoscale semiconductor devices \cite{Jun_transp}, where quantum-mechanical effects must be taken into account. In this context the dissipative term $-\xi\rho u$ describes collisions between electrons and the
semiconductor crystal lattice (see, for instance, \cite{BW}), and $\tau=1/\xi$ is the relaxation time. 
The advantage of using macroscopic models for quantum fluids, with respect to kinetic models, is their reduced complexity, especially from a computational point of view \cite{MRS}. Moreover, hydrodynamic models correctly describe high field phenomena or submicronic devices.
However, in certain regimes, as in particular for low carrier densities and small electric fields, these models can be further reduced to some simpler ones. In the context of semiconductor devices for instance, quantum transport of electrons can be effectively described by the quantum drift-diffusion (QDD) equation \cite{VanR}, given by
\begin{equation}\label{eq:qdd_intro}
\begin{aligned}
&\d_t\rho + \dive \left(2\rho \nabla\left( \frac{\Delta \sqrt{ \rho}}{\sqrt{\rho}}\right) - \nabla \rho^{\gamma} -\rho\nabla V\right) = 0\\
&-\Delta V=\rho-g.
\end{aligned}
\end{equation}
The (QDD) equation can be formally recovered from system \eqref{eq:qns_intro} as a relaxation limit. Precisely,  by rescaling the time  as follows
\begin{equation}\label{eq:scaling}
t'=\epsilon t,\quad(\rho^\epsilon, u^\epsilon)(t', x)=\left(\rho, u\right)\left(\frac{t'}{\epsilon}, x\right),
\end{equation}
where $\epsilon:=1/\xi$, the scaled system reads
\begin{equation}\label{eq:Corrado}
\begin{aligned} 
& \partial_t\re + \frac{1}{\epsilon} \dive(\re\ue ) = 0 \\ 
& \partial_t(\re\ue) + \frac{1}{\epsilon} \dive (\re\ue \otimes \ue) - \frac{1}{\epsilon}\dive (\re D\ue) + \frac{1}{\epsilon} \nabla \re^{\gamma} + \frac{1}{\e} \re \nabla V_{\e}= \frac{1}{\epsilon}2\re\nabla \left( \frac{\Delta \rre}{ \rre} \right) - \frac{1}{\epsilon^2} \re\ue\\
&- \Delta V_{\e} = \re - g.
\end{aligned} 
\end{equation}
Thus in the limit $\epsilon\to0$, we formally obtain that 
\begin{equation}\label{eq:lim_mom}
\lim_{\epsilon\to0}\re\frac{\ue}{\epsilon}=2\rho\nabla\left(\frac{\Delta\sqrt{\rho}}{\sqrt{\rho}}\right)-\nabla\rho^\gamma- \rho\nabla V
\end{equation}
and therefore the (QDD) equation.

The main purpose of our paper is to rigorously prove the above limit, that is to prove that scaled finite energy weak solutions to \eqref{eq:qns_intro} converge to finite energy weak solutions to \eqref{eq:qdd_intro}. To this aim, in the following we shall refer to \eqref{eq:Corrado} with initial datum $(\rho^0,u^0)$ and doping profile $g$ possibly depending in a suitable way on the relaxation parameter $\e$ as well.
\begin{theorem}\label{teo:mainintro}
Let $(\re,\ue, V_{\e})$ be a weak solution of $\eqref{eq:qns_intro}$ in the sense of Definition \ref{WSQNS} with data $(\rho^0_\epsilon, u^0_\epsilon, g_\e)$ satisfying 
\begin{equation*}
\begin{aligned}
&\{\rho_\epsilon^0\}_{\e}\mbox{ is bounded in }L^1 \cap L^{\gamma} (\T)\ \hbox{such that}\ \rho_\epsilon^0 \to \rho^0\ \hbox{in}\ L^q(\T),\, q<3 \\
&\{\nabla \sqrt{\rho_\epsilon^0}\}_{\e}\mbox{ is bounded in }L^2(\T),\\
&\{\sqrt{\rho_\epsilon^0 }u_\epsilon^0\}_{\e}\mbox{ is bounded in }L^2(\T),\\
&\{g_{\e}\}_{\e}\mbox{ is bounded in }L^2(\T)\ \hbox{such that}\ g_{\e}\weakto g\ \hbox{in}\ L^2(\T).
\end{aligned}
\end{equation*}
	Then, up to subsequences, there exists $\rho\geq 0$ and $V$ such that 
	\begin{align*}
	&\rre\rightarrow \rrho\textrm{ strongly in }L^{2}((0,T);H^{1}(\T))\\
	& \nabla V_{\e} \rightarrow \nabla V \textrm{ strongly in } C([0,T); L^2(\T)),
	\end{align*}
	 and $(\rho, V)$ is a finite energy weak solution of $\rho$ of \eqref{eq:qdd_intro} with initial datum $\rho(0)=\rho^0$, in the sense of  Definition \ref{WSGF}. Namely there exist $\Lambda,\,\mathcal{S} \in L^{2}((0,T)\times\T)$ such that 
\begin{align}\label{eq:Lambda}
\rrho\Lambda&=2\dive (\rrho\nabla^2\rrho-\nabla\rrho\otimes\nabla\rrho-\rho^{\gamma}\mathbb{I})- \rho \nabla V \textrm{ in }\mathcal{D}'((0,T)\times\T)\\
\rrho\mathcal{S}&=2\rrho\nabla^2\rrho-2\nabla\rrho\otimes\nabla\rrho\,\mbox{ a.e. in }(0,T)\times\T\label{eq:S}
\end{align}
and $C>0$ such that for $a.e.$ $t\in(0,T)$
\begin{equation}\label{eq:fisher}
\int_{\T}\left(|\nabla\rrho|^2+\frac{\rho^{\gamma}}{\gamma-1}+ \frac{1}{2}|\nabla V|^2\right)(t)dx + \int_0^t \int_{\T} |\Lambda|^2dsdx \leq C
\end{equation} 
\begin{equation}\label{eq:free}
\begin{aligned}
\int_{\T}(\rho(\log\rho-1)+1)(t)dx& + \int_0^t \int_{\T}|\mathcal{S}|^2dsdx +\frac{4}{\gamma}\int_0^t \int_{\T}|\nabla\rho^{\frac{\gamma}{2}}|^2dsdx\\
& +\int_0^t \int_{\T} \rho(\rho-g)\;dsdx \leq \int_{\T}(\rho^0 (\log\rho^0 -1)+1)dx. 
\end{aligned}  
\end{equation}
Moreover, if in addition, the initial data also satisfy
\begin{equation*}
\begin{aligned}
&\sqrt{\rho_\e^0}\ue^0\to 0\mbox{ strongly in }L^{2}(\T)\\
&\nabla \sqrt{\rho_\e^0} \to \nabla\sqrt{\rho^0}\mbox{ strongly in }L^{2}(\T)\\
&\rho_\e^0 \to\rho^0\mbox{ strongly in } L^{\gamma}(\T), \\
\end{aligned}
\end{equation*}
then $\rho$ is an energy dissipating   weak solution, meaning that in addition to be a finite energy weak solution for a.e. $t\in(0,T)$ it holds
\begin{equation*}
\begin{aligned}
\int_{\T}\left(|\nabla\rrho|^2+\frac{\rho^{\gamma}}{\gamma-1}+ \frac{1}{2}|\nabla V|^2\right)(t)dx & +\int_0^t\int_{\T}|\Lambda|^2dsdx \leq \\
& \int_{\T}\,|\nabla\sqrt{\rho^0}|^2dx+\int_{\T}\frac{{(\rho^0)^\gamma}}{\gamma-1}dx+\int_{\T} \frac{1}{2}|\nabla V(0)|^2dx.
\end{aligned}
\end{equation*} 
\end{theorem}
Let us notice that the estimates \eqref{eq:fisher} and \eqref{eq:free} yield the boundedness of the Fisher entropy and the free energy, respectively. On the other hand, the quantities $\Lambda$ and $\mathcal S$ characterized in \eqref{eq:Lambda} and \eqref{eq:S} provide the associated entropy dissipations, in a weaker sense than the estimates derived in \cite{GST} and \cite{JM}. Indeed, formally
\begin{equation}\label{eq:lams}
\begin{aligned}
\Lambda= 2\rrho\nabla\left(\frac{\Delta\sqrt{\rho}}{\sqrt{\rho}}\right)-\frac{\nabla\rho^\gamma}{\rrho},\qquad \mathcal S = \rrho\nabla^2\log\rho,
\end{aligned}
\end{equation}
but, due to the low regularity setting and the possible presence of vacuum regions it seems not possible to obtain the relations \eqref{eq:lams} in the limit, so the only available information we have is given by formulas \eqref{eq:Lambda} and \eqref{eq:S}. We refer to Remark \ref{rem:1} and Proposition \ref{enebdqddreg} below for more details on the tensor $\Lambda$ and $\mathcal S$.
%\begin{remark}
%We remark that, in the modeling of semiconductor devices, the dynamics is also subject to a self-consistent electrostatic potential, determined by the Poisson equation, so that the system actually reads
%\begin{equation}\label{eq:qns_V}
%\begin{aligned} 
%& \partial_t\rho + \dive(\rho u ) = 0 \\ 
%& \partial_t(\rho u) + \dive (\rho u \otimes u) - \dive (\rho Du) + \nabla \rho^{\gamma} +\rho\nabla V= 2\rho\nabla \left( \frac{\Delta \sqrt{\rho}}{ \rho} \right) - \xi \rho u\\
%&-\Delta V=\rho-C,
%\end{aligned} 
%\end{equation} 
%where $C=C(x)$ is given and represents the doping concentration. In this case the relaxation limit would give 
%\begin{equation*}\begin{aligned}
%&\d_t\rho + \dive \left(2\rho \nabla\left( \frac{\Delta \sqrt{ \rho}}{\sqrt{\rho}}\right) - \nabla \rho^{\gamma} -\rho\nabla V\right) = 0\\
%&-\Delta V=\rho-C.
%\end{aligned}\end{equation*}
%We remark that our main Theorem straightforwardly applies also to \eqref{eq:qns_V}. We chose not to include the electrostatic potential in our subsequent analysis only for the sake of simplifying the exposition. Indeed this additional term can be treated in a similar fashion as the pressure term.
%\end{remark}

An exhaustive list of all references concerning diffusive relaxation limits and asymptotic behavior for systems of conservation laws with friction, and in particular for hydrodynamic models for semiconductors, is beyond the interest of our presentation. For the theory of diffusive relaxation, we refer here to  \cite{DMTrans}, concerning in particular the case of multidimensional general semilinear systems, and the reference therein.  Moreover, concerning in particular the case of high friction limits with relative entropy techniques in the context of 
Korteweg theories \cite{GLT17}, we refer to  \cite{LT17,CCL}; see also \cite{LT13} for the case of Euler equations with friction.
Finally, for the particular case of Euler--Poisson models for semiconductors, we recall that the rigorous analysis of the diffusive
 relaxation limits  in the context of weak, entropic solutions  started with the seminal paper \cite{MN}, where the one dimensional case is treated using compensated compactness; see also  \cite{LM_rel,Latt} for the multi--$d$ case. 
 
Besides the modeling point of view, there are some other mathematical aspects which motivate our result.
%for which we shall  not refer to the presence of the potential $V$, namely, we are referring to the Navier-Stokes-Korteweg model.
First of all, the study of this singular limit is related to the asymptotic behavior of solutions to \eqref{eq:qns_intro} for large times. Let $(\rho^\star, u^\star)=(r, 0)$ be the stationary solution to \eqref{eq:qns_intro}, where, for $g$ constant, $r=\fint\rho = g$ is the mean value of the particle density. Then it can be shown that solutions to \eqref{eq:qns_intro} exponentially converge towards $(\rho^\star, u^\star)$ as $t\to\infty$, see \cite{GJT} for the one-dimensional problem (with suitable boundary conditions, see also \cite{LZ} for some extensions) and \cite{BGLVV} for the proof of this result in the framework of finite energy weak solutions in the three dimensional torus. 

On the other hand, it is also interesting to determine the asymptotic dynamics which governs the exponential convergence to equilibrium. This is indeed achieved by performing the scaling in \eqref{eq:scaling}, hence the (QDD) equation \eqref{eq:qdd_intro} also gives the asymptotic dynamics we are interested in.

On a related subject, let us also comment on the inviscid counterpart of system \eqref{eq:qns_intro}, namely the quantum hydrodynamic (QHD) system \cite{AM, AM_CM}. Due to the dissipative term $-\xi\rho u$, also in this case it is possible to show both the exponential convergence towards the stationary solution \cite{HLM, HLMO} and the relaxation limit \cite{JLM_rel}, again towards the (QDD) equation. However the only available results here deal with small, regular perturbations around stationary solutions. This can be seen as due to the lack of regularizing effect of the viscosity, by means of the BD entropy estimates.

Notice that in Theorem \ref{teo:mainintro} the only assumption needed is the initial energy associated to the system \eqref{eq:qns_intro} to be uniformly bounded at the initial time. In particular no assumptions on the limiting solution to \eqref{eq:qdd_intro} are given. Consequently, as a byproduct our main Theorem also provides an alternative proof for the existence of finite energy weak solutions to \eqref{eq:qdd_intro}, see \cite{JM} and \cite{GST}. Furthermore, in the proof of our main result it is possible to see how the energy and BD entropy estimates, respectively, associated to \eqref{eq:qns_intro} yield, in the limit $\varepsilon\to0$, the Fisher and free energy, respectively, associated to \eqref{eq:qdd_intro}. Those facts were already noticed, in a similar context, in the recent preprint \cite{Blub}, see also \cite{BlubCR}. More precisely, the Authors in \cite{Blub} consider the one dimensional shallow water equations with a nonlinear damping term. By using a similar scaling as in \eqref{eq:scaling}, the authors study the convergence towards a lubrication type model. In particular in \cite{Blub} the Authors emphasize how the BD entropy for the hydrodynamical system converge towards the so called Bernis-Friedman \cite{BF} entropy, associated to the limiting diffusive equation. 

Let us remark that also in the context of semiconductor device modeling it would make sense to consider a nonlinear damping term, as in \cite{Blub}. Indeed this would correspond to the case when the relaxation time $\tau$ is no longer a constant, but a function of the particle density. This is consistent with the derivation of the hydrodynamic system from kinetic theories, as in general the relaxation coefficient may depend on the particle density. 
%{\color{red}We stress that also our result can be  adapted to the case of nonlinear damping, provided }
Finally, our result can also be seen as related to the derivation of \eqref{eq:qns_intro} and \eqref{eq:qdd_intro} from kinetic equations. These macroscopic models for quantum transport are usually derived from collisional Wigner-type equations, with a suitable choice of the collision operator, see  \cite{DR03,DMR05} and \cite{Jun_rev} for a more comprehensive discussion about those issues. In particular, the QNS system with a linear damping was derived in \cite{JM_fQNS} by applying the moment method to a Wigner-type equation whose collisional operator is chosen to be the sum of a BGK and a Caldeira-Leggett-type operator, see also \cite{JLM_WFP} where an alternative derivation is given by avoiding the Chapman-Enskog expansion. Actually in \cite{JM_fQNS} and \cite{JLM_WFP} the authors derive the full QNS system, where also the dynamics of the energy density is given, in our paper we only consider the isentropic dynamics given by \eqref{eq:qns_intro}. We also mention \cite{BM} where the QNS system without damping is derived.
On the other hand the QDD equation can also be derived from the same Wigner-type equation by using a diffusive scaling. 
In this sense our result, obtained by using the scaling \eqref{eq:scaling}, can be seen as linking the two different scalings used to derive \eqref{eq:qns_intro} and \eqref{eq:qdd_intro} directly from kinetic models.

\subsection*{Organization of the paper}

The paper is organized as follows. In Section \ref{sec:2} we give the definition of weak solution for the Quantum-Navier-Stokes system and the Quantum-Drift-Diffusion equation, we give a formal proof of the $\e$-independent estimates and we state the main theorem and in Section \ref{sec:3} we prove the main result of the paper.
\subsection*{Notations}
We denote with $L^p(\mathbb{T}^n)$ the standard Lebesgue spaces. The Sobolev space of $L^p$ functions with $k$ distributional derivatives in $L^p$ is denoted $W^{k,p}$, in the case $p=2$ we write $H^k(\mathbb{T}^n)$. The spaces $W^{-k,p}$ and $H^{-k}$ denote the dual spaces of  $W^{k,p'}$ and $H^k$ where $p'$ is the usual H\"{o}lder conjugate of p. Given a Banach space $B$, the classical Bochner space of real valued function with values in $B$ is denote by $L^{p}(0,T;B)$ and sometimes also the abbreviation $L^{p}_t(B_x)$ will be used. Given a function $f \in L^p(\T)$ we denote the average of $f$ alternatively $\fint f$ or $\bar f$, and 
throughout the paper we can assume without loss of generality that $|\T|=1$. We denote with $Du = (\nabla u + \nabla u^t)/2$ the symmetric part of the Jacobian matrix $\nabla u$ and with $Au = (\nabla u -\nabla u^t)/2$ the antisymmetric part. Finally, the subscript $\e$ used to denote sequences of functions has to be always understood running over a countable set.

\section{Definition of Weak Solutions and Main Result}\label{sec:2}
In this section we give the definition of weak solutions for the system \eqref{qns} and \eqref{eq:qdd_intro}. The existence of such solutions is out of the main aims of this paper;  we underline here this result can be proved  via a compactness argument  as done in \cite{LV, AS1} with minor modifications; see also \cite{AS,AS3,ASnew}. However, it is worth observing that with this technique, thanks to the particular choice of the   approximating sequence,  the constructed weak solutions verify various appropriate bounds, instrumental for our convergence result. For this reason,  we shall assume these bounds are valid in our framework; see Definition \ref{WSQNS} and Remarks \ref{rem:1} and \ref{rem:2} below.

\subsection{Weak solutions of the Quantum-Navier-Stokes-Poisson equations}
Let us consider the following system in $(0,T)\times\T$ for a given  $g_{\e}: \T \rightarrow \R$,
\begin{equation}\label{qns}
\begin{aligned} 
& \partial_t\re + \frac{1}{\epsilon} \dive(\re\ue ) = 0 \\ 
& \partial_t(\re\ue) + \frac{1}{\epsilon} \dive (\re\ue \otimes \ue) - \frac{1}{\epsilon}\dive (\re D\ue) + \frac{1}{\epsilon} \nabla \re^{\gamma} + \frac{1}{\e} \re \nabla V_{\e}= \frac{1}{\epsilon}2\re\nabla \left( \frac{\Delta \rre}{ \rre} \right) - \frac{1}{\epsilon^2} \re\ue\\
&- \Delta V_{\e} = \re - g_{\e}
\end{aligned} 
\end{equation} 
with initial data 
\begin{equation}\label{eq:id}
\begin{split}
%& V_\e(0,x) = V_\e^0(x) \text{ such that } - \Delta V_\e^0= \re^0- g_\e \\
& \rho_{\epsilon} (0,x) = \rho^0_{\epsilon}(x), \\
& \rho_{\epsilon}(0,x) u_{\epsilon}(0,x) = \rho^0_{\epsilon}(x) u^0_{\epsilon}(x),
\end{split}
\end{equation}
on $ \{t=0\}\times \mathbb{T}^3$ and zero average condition for $V_\e$, namely
\begin{equation}\label{meanV}
\fint_{\T} V_\e(x,t) dx = 0. 
\end{equation}
%As $V_\e(t,x)$ depends parametrically on $t$ in view of the right hand side of \eqref{qns}$_3$ we define $V_\e^0(x)$ as the solution of that elliptic equation for $\re^0- g_\e$. 
We emphasize that in order to solve \eqref{qns}$_3$,  \eqref{meanV} we need the  following assumption on the doping profile  $g_\e$:  
$\fint_{\T} g_\e(x) dx = M_\e$ and $M_\e > 0$ where 
$M_\e:=\fint_{\T} \re^0(x)$; see also Remark \ref{rem:4} below.

The definition of weak solution is the following.
\begin{definition}\label{WSQNS}
	Given $\rho^0_{\epsilon}$ positive and such that $\rrho^0_{\epsilon}\in H^{1}(\T)$ and $\rho^0_{\epsilon}\in L^{\gamma}(\T)$, $g_\e \in L^2(\T)$ and $u^0_{\epsilon}$ such that $\rrho^0_{\epsilon}u^0_{\epsilon}\in L^{2}(\T)$, then $(\re, \ue,  V_{\e})$ with $\re\geq0$ and $V_\e$ with zero average is a weak solution of the Cauchy problem \eqref{qns}-\eqref{eq:id} if the following conditions are satisfied. 
	\begin{enumerate}
		\item Integrability conditions:
		\begin{align*}
		&\rre\in L^{\infty}(0,T;H^{1}(\T))\cap L^{2}(0,T;H^2(\T)),
		&\rre\ue\in L^{\infty}(0,T;L^{2}(\T)),\\
		&\re^{\gamma}\in L^{\infty}(0,T;L^1(\T)),
		&\MT_{\e}\in L^{2}(0,T;L^{2}(\T)), \\
		&\re\in C(0,T;L^2(\T)), & V_{\e} \in C(0,T;H^{2}(\T)).
		\end{align*}
		\item Continuity equation:\\
		For any $\phi\in C^{\infty}([0,T)\times\T;\R)$
		\begin{equation}\label{eq:wfce}
		\int\re^0\phi(0)dx+\iint\re\phi_t+\frac{1}{\e}\rre\rre \ue\nabla\phi \; dsdx=0.
		\end{equation}
		\item Momentum equation:\\
		For any $\psi\in  C^{\infty}([0,T)\times\T;\R^3)$
		\begin{equation}\label{eq:wfme}
		\begin{aligned}
		&\int\re^0\ue^{0}\psi(0)dx +\iint\rre(\rre \ue)\psi_t \;dsdx + \frac{1}{\e}\iint\rre \ue\rre \ue\cdot\nabla\psi \; dsdx\\
		&-\frac{1}{\e}\iint\rre\MT_{\e}^{s}:\nabla\psi \; dsdx+ \frac{1}{\e}\iint\re^{\gamma}\dive\psi \; dsdx - \frac{1}{\e}\iint \re \nabla V_{\e} \psi \; dsdx \\
		&-\frac{1}{\e}\iint 2 \rre\,\nabla^2\rre:\nabla\psi \; dsdx+ \frac{1}{\e}\iint 2 \nabla\rre\otimes\nabla\rre\,:\nabla\psi \;dsdx +\frac{1}{\e^{2}}\iint\re\ue\psi\; dsdx =0.
		\end{aligned}
		\end{equation}
		\item Poisson equation:\\
		For a.e. $(t,x)\in (0,T)\times\T$ it holds that 
		\begin{equation}\label{eq:poisson}
		-\Delta V_{\e}=\re-g_{\e}.
		\end{equation}
		\item Energy dissipation:\\
		For any $\varphi\in C^{\infty}([0,T)\times\T;\R)$
		\begin{equation}\label{eq:dissor}
		\iint\rre\MT_{\e,i,j}\varphi \; dsdx =-\iint\re u_{\e,i}\nabla_{j}\varphi \; dsdx -\iint2\rre u_{\e,i}\otimes\nabla_{j}\rre\,\varphi  \;dsdx.
		\end{equation}
		\item Energy Inequality:\\
		For a.e. $t\in(0,T)$
		\begin{equation}\label{fes-1}
		\begin{split}
		&\int_{\T} \left(\frac{1}{2} \re|\ue|^2 + \frac{\re^\gamma}{\gamma - 1}+2 |\nabla \sqrt{\re}|^2+ \frac{1}{2} |\nabla V_{\e}|^2 \right)(t) dx +\frac{1}{\epsilon} \int_0^t\int_{\T} |\MT^s_{\e}|^2 dsxd +\frac{1}{\epsilon^2}\int_0^t\int_{\T} \re|\ue|^2 dsdx \\
		& \leq \int_{\T} \left(  \frac{1}{2} \re^0|\ue^0|^2 + \frac{({\re^{0})}^\gamma}{\gamma - 1} +  2 |\nabla \sqrt{\re^0}|^2 +\frac{1}{2} |\nabla V_{\e}(0)|^2\,\right)dx.
		\end{split}
		\end{equation}
		\item BD Entropy: Let $w_\epsilon=u_\epsilon+\nabla\log\rho_\epsilon$,  then there exists a tensor $\mathcal{S}_{\epsilon}\in L^{2}(0,T;L^{2}(\T))$ such that 
\begin{equation}\label{eq:dispdiss}
\rre\mathcal{S}_{\epsilon}=2\rre\nabla^2\rre-2\nabla\rre\otimes\nabla\rre\,\mbox{ a.e. in }(0,T)\times\T	
\end{equation}		
we have for a.e. $t\in(0,T)$
		\begin{align}
		&\nonumber\int_{\T}\left(\frac{1}{2}\epsilon \re |\wn|^2 + \epsilon\frac{\re^{\gamma}}{\gamma - 1}+\epsilon 2|\nabla \sqrt{\re}|^2 + \frac{1}{2}\e |\nabla V_{\e}|^2 + (\re(\log\re -1)+1)\right)(t) dx\\
		&  +\int_0^t\int_{\T} |\MT_{\e}^{a}|^2 \; dsdx +\int_0^t\int_{\T}|\mathcal{S}_{\epsilon}|^2 \; dsdx +\frac{4}{\gamma}\int_0^t\int_{\T}|\nabla \re^{\frac{\gamma}{2}}|^2 \; dsdx + \frac{1}{\epsilon}\int_0^t\int_{\T} \re |\ue|^2 \; dsdx \nonumber\\
		& \nonumber +\int_0^t\int_{\T}  \re(\re- g_\e) \;dsdx 
		\leq \\
		& \label{newbdestimate1} \int_{\T} \left( \frac{1}{2}\epsilon \re^0 |\wn^0|^2 + \epsilon \frac{({\re^{0}})^{\gamma}}{\gamma - 1} + \epsilon 2|\nabla \sqrt{\re^0}|^2 + + \frac{1}{2}\e |\nabla V_{\e}(0)|^2+ (\re^0(\log \re^0 -1) + 1 ) \right)dx.
\end{align}
\item There exists an absolute constant $C$ such that 
\begin{equation}\label{eq:jung}
\begin{aligned}
\int_0^t \int_{\T} |\nabla^2\rre|^2+|\nabla\re^{\frac{1}{4}}|^4& dsdx\leq C\int_{\T} \left( \frac{1}{2}\epsilon \re^0 |\wn^0|^2 + \epsilon \frac{(\rho_{\e}^0)^{\gamma}}{\gamma - 1}+\epsilon 2|\nabla \sqrt{\re^0}|^2 \right)dx\\
&+C\int_{\T} \left(\frac{1}{2}\e |\nabla V_{\e}(0)|^2+ (\re^0(\log \re^0 -1) + 1 ) \right)dx
\end{aligned}
\end{equation}
\end{enumerate}
\end{definition}

The following remarks aim at explaining some peculiar points of the definition of weak solutions. In particular, we explain   the presence of the tensors $\MT_{\e}$ and $\mathcal{S}_{\epsilon}$ in the following remarks, and the reason why we must further assume the bound \eqref{eq:jung}.

 %Finally, we explain   the presence of the tensors $\MT_{\e}$ and $\mathcal{S}_{\epsilon}$ in the following remarks, showing in particular how they arise form the aforementioned approximating procedure, and the reasons why we can further assume the bound \eqref{eq:jung}.

\begin{remark}[\emph{Weak formulation of the quantum term}]\label{rem:3orderlimit}
We emphasize that in the weak formulation introduced above, the third-order term in the momentum equation can be written in different ways:
\begin{equation}\label{eq:dispide}
2\rho \nabla\left( \frac{\Delta \sqrt{ \rho}}{\sqrt{\rho}}\right)= \dive (\rho \nabla^2 \log \rho) = \dive ( 2 \sqrt{\rho} \nabla^2 \sqrt{\rho} - 2 \nabla\sqrt{\rho}\otimes \nabla \sqrt{\rho}),
\end{equation} 
and we are using the last expression to give a distributional meaning to the third-order tensor term. 
\end{remark}
\begin{remark}[\emph{The velocity field and the vacuum set}]
We stress that in the definition of weak solutions the vacuum region can be of positive measure. As a consequence, the velocity it is not uniquely defined in the vacuum set, namely if we change its value on $\{\rho=0\}$ we would still have the same weak solution of \eqref{qns}-\eqref{eq:id}. Moreover, even   if we choose $u=0$ on $\{\rho=0\}$ we can not deduce any a priori bound in any Lebesgue space.
\end{remark}

\begin{remark}[\emph{About the tensors $\MT_{\e}$ and $\mathcal{S}_{\epsilon}$}]\label{rem:1}

The presence of the tensor $\MT_{\e}$ is   due to the possible presence of vacuum regions. Indeed, if the density is bounded away from zero, \eqref{eq:dissor} implies that $\MT_{\e}=\rre\nabla\,\ue$. On the other hand, even in the case when the vacuum has zero Lebesgue measure, $\nabla \ue$ also can not be defined in a distributional sense. Therefore the tensor $\MT_{\e}$ arises as a weak $L^{2}$-limit of the sequence $\{\sqrt{\rho^n}\nabla\ue^n)\}_{n}$, where $\{\re^n,\ue^n\}_n$ is a suitable sequence of approximations, see \cite{LV, AS1}. Analogously,   $\mathcal{S}_{\epsilon}$ is again motivated   by the presence of vacuum regions. Indeed, also in this case, if the density is bounded away from zero, by using the identity \eqref{eq:dispide}, we have that $\mathcal{S}_{\epsilon}=\rre\nabla^2\log\re$. As for $\MT_{\epsilon}$, the tensor $\mathcal{S}_{\e}$ arises as an $L^2$ weak limit of the sequence $\{\sqrt{\rho^n}\nabla^2\log\re^n\}_{n}$, where $\{\re^n,\ue^n\}_n$ is a suitable sequence of approximations. Notice that the fact that $\{\sqrt{\rho^n}\nabla^2\log\re^n\}_{n}$ is bounded in $L^{2}((0,T)\times\T)$ because of the BD Entropy, see Proposition \ref{prop:energybd}. 
\end{remark}
\begin{remark}[\emph{About the inequality \eqref{eq:jung} }]\label{rem:2} 
Regarding \eqref{eq:jung} we recall the following inequality proved in \cite{JM} : There exists $C>0$ depending only on the dimension, such that for any function $\rho\geq 0$ with $\rrho\in H^{2}(\T)$
\begin{equation}\label{nablalogrho}
\iint|\nabla\rho^{\frac{1}{4}}|^4 \; dsdx +\iint|\nabla^2\rrho|^2 \; dsdx \leq C \iint\rho|\nabla^2\log\rho|^2 \; dsdx.
\end{equation}
Therefore, \eqref{eq:jung} follows by applying \eqref{nablalogrho} to some approximation $\re^n$, for which the quantity $\sqrt{\rho^n}\nabla^2\log\re^n$ is well-defined, using the bound on $\sqrt{\rho^n}\nabla^2\log\re^n$ inferred from the BD Entropy and using the weak lower-semicontinuity of the norms. For completeness we give the proof of \eqref{eq:jung} in the appendix.
\end{remark}
\begin{remark}[\emph{About the integrability conditions and the Poisson equation}]\label{rem:4}
We notice that by using the integrability hypothesis and \eqref{eq:wfce} the average of the density is conserved, namely, $\fint_{\T} \re(t,x)\,dx=M_{\e}$ for any $t\in(0,T)$. Therefore, $\re-g$ has zero average for any $t\in(0,T)$ and then the compatibility for the Poisson equation $\eqref{qns}_3$ is always satisfied. \\ Next, from the mimimun regularity required for $\rho_\e$ to be a weak, finite energy solution, we readily obtain $\re\in C([0,T];L^{2}(\T))$; this (redundant) condition is thus listed explicitly in Definition \ref{WSQNS}. Hence $\re|_{t=0}=\re^{0}$ in $L^{2}(\T)$ and we have that $V_{\e}(0)$ is well-defined and coincides with the solution of the corresponding elliptic equation  at $t=0$, that is:
\begin{equation*}
-\Delta\,V_\e(0)=\re^0-g_{\e}.
\end{equation*}
Finally,  since $\re^0$ and $g_{\e}$ are both bounded in $L^{2}(\T)$ we have that $V_{\e}|_{t=0} = V_\e(0)$ is bounded in $H^{2}(\T)$.
\end{remark}

Weak solutions of \eqref{qns}-\eqref{eq:id} can be obtained by using the same approximation argument of \cite{LV} with minor changes. In particular, being the proof of the existence based on a compactness argument, \eqref{fes-1} and \eqref{newbdestimate1} should be proved for suitable approximate solutions and then would be obtained trough a limiting argument. On the other hand, as usual in PDE they can be motivated by formal estimate. This is exactly the content of the next proposition.  
\begin{proposition}\label{prop:energybd}
	Let  $\e>0$ and $(\re,\ue, V_{\e})$ be a sequence of smooth solutions of \eqref{qns} with initial data \eqref{eq:id}. Define 
	$\wn=\ue+\nabla\log\re$. Then, for any $t\in[0,T)$ the pair $(\re, \ue, V_\e)$ satisfies the following estimates:  
	\begin{enumerate}
		\item (Energy Estimate)
		\begin{equation}\label{fes}
		\begin{aligned}
		&\int_{\T} \left(\frac{1}{2} \re|\ue|^2 + \frac{\re^\gamma}{\gamma - 1}+2 |\nabla \sqrt{\re}|^2+ \frac{1}{2}|\nabla V_{\e}|^2\right)(t) dx +\frac{1}{\epsilon} \int_0^t\int_{\T} \re |D\ue|^2 \; dsdx +\frac{1}{\epsilon^2}\int_{0}^{t}\int_{\T} \re|\ue|^2\; dsdx  \\
	    & = \int_{\T} \left(  \frac{1}{2} \re^0|\ue^0|^2 + \frac{({\re^{0})}^{\gamma}}{\gamma - 1} +  2 |\nabla \sqrt{\re^0}|^2+ \frac{1}{2}|\nabla V_{\e}(0)|^2 \right)dx.
		\end{aligned}
		\end{equation}
		\item (BD Entropy)
		 \begin{equation}\label{newbdestimate}
	       \begin{aligned}
	    &\int_{\T} \left(\frac{1}{2}\epsilon \re |\wn|^2 + \epsilon\frac{\re^{\gamma}}{\gamma - 1}+\epsilon 2|\nabla \sqrt{\re}|^2 + \e \frac{1}{2}|\nabla V_{\e}|^2 + (\re(\log\re -1)+1)\right)(t)dx+
	    \int_0^t\int_{\T} \re |A\ue|^2\; dsdx  \\
	    &+ \int_0^t\int_{\T} \re |\nabla^2 \log \re|^2 \; dsdx + \frac{4}{\gamma} \int_0^t\int_{\T} |\nabla \re^{\frac{\gamma}{2}}|^2\; dsdx + \int_0^t\int_{\T} \re(\re - g_\e) \;dsdx  + \frac{1}{\epsilon} \int_0^t\int_{\T} \re |\ue|^2 \; dsdx \\
	    & = \int_{\T} \left( \frac{1}{2}\epsilon \re^0 |\wn^0|^2 + \epsilon \frac{({\re^{0}})^{\gamma}}{\gamma - 1} + \epsilon 2|\nabla \sqrt{\re^0}|^2 + \frac{1}{2}\e|\nabla V_{\e}(0)|^2+ (\re^0(\log \re^0 -1) + 1 ) \right)dx.
	      \end{aligned} 
	    \end{equation}
	\end{enumerate}
\end{proposition}
\begin{proof}
Let us start by  recalling the following alternative ways to write the dispersive term:
\begin{equation*}
2\rho \nabla\left( \frac{\Delta \sqrt{ \rho}}{\sqrt{\rho}}\right)= \dive (\rho \nabla^2 \log \rho) = \dive ( 2 \sqrt{\rho} \nabla^2 \sqrt{\rho} - 2 \nabla\sqrt{\rho}\otimes \nabla \sqrt{\rho}). 
\end{equation*} 
Then, for smooth solutions $(\rho_\e,u_\e)$ of $\eqref{qns}$, as usual we multiply the continuity equation $\eqref{qns}_1$ by $\gamma \rho_\e^{\gamma-1} / (\gamma -1)$ and the momentum equation $\eqref{qns}_2$ by $u_\e$ and integrate in space to get:
\begin{equation}\label{es}
\begin{aligned}
&\frac{d}{dt}\int_{\mathbb{T}^3} \left(\frac{1}{2} \rho |u_\e|^2 \right) dx  +  \frac{1}{\epsilon} \int_{\mathbb{T}^3}\rho |Du_\e|^2 dx  + \frac{1}{\epsilon} \int_{\mathbb{T}^3} \nabla \rho^{\gamma} \cdot  u_\e dx + \frac{1}{\epsilon} \int_{\mathbb{T}^3} \rho_\e \nabla^2 \log \rho_\e : \nabla u_\e dx 
\\
&\ +  \frac{1}{\epsilon^2} \int_{\mathbb{T}^3} \rho |u_\e|^2 dx + \frac{1}{\e}\int_{\mathbb{T}^3} \re \nabla V_{\e} \ue dx = 0,
\end{aligned}
\end{equation}
\begin{equation}\label{potential}
\frac{d}{dt} \int_{\T} \frac{1}{2}|\nabla V_{\e}|^2\;dx = -\int_{\T} (\Delta V_{\e})_t V_{\e} \;dx  = \int_{\T} \partial_t \re V_{\e} \;dx  = - \frac{1}{\e} \int_{\T} \dive(\re \ue) V_{\e} \;dx = \frac{1}{\e}\int_{\T} \re \nabla V_{\e} \ue \;dx  ,
\end{equation}
\begin{equation}\label{press}
\frac{d}{dt} \int_{\T} \frac{\rho_\e^\gamma}{\gamma - 1} \; dx - \frac{1}{\epsilon} \int_{\T} \nabla \rho_\e^\gamma u_\e dx = 0.
\end{equation}
Moreover, using the equation satisfied by the so called \textit{effective velocity} $\nabla \log \rho_{\epsilon}$, (see \cite{BVY} for its formal derivation) 
\begin{equation*}
(\rho_\e \nabla \log \rho_\e)_t + \frac{1}{\epsilon}\dive (\rho_\e \nabla \log \rho_\e \otimes u_\e) + \frac{1}{\epsilon} \dive (\rho_\e ^t \nabla u_\e) = 0,
\end{equation*}
after multiplication by $\nabla \log \rho_{\epsilon}$ itself and space integration  we obtain 
\begin{equation}\label{effqns}
\frac{d}{dt} \int_{\T} 2 |\nabla \sqrt{\rho}_\e|^2 \; dx - \frac{1}{\epsilon} \int_{\T} \rho_\e \nabla^2 \log \rho_\e : \nabla u_\e \;dx = 0.
\end{equation}
Hence summing $\eqref{es}, \eqref{press}$ and $\eqref{effqns}$ and integrating time we end up to 
\eqref{fes}.
For the BD Entropy relation, we take once again advantage of  the effective velocity, and in particular introducing the quantity $w_\e = u_\e + \nabla \log \rho_\e$.
Then, using the relations 
\begin{align*}
\frac{1}{\epsilon} \dive (\rho_\e \nabla \log \rho_\e) &= \frac{1}{\epsilon}\Delta \rho_\e,
\\
(\rho_\e \nabla \log \rho_\e)_t &= - \frac{1}{\epsilon} \nabla \dive (\rho_\e u_\e), 
\\ 
\frac{1}{\epsilon} \dive (\rho_\e u_\e \otimes \nabla \log \rho_\e +  \rho_\e \nabla \log \rho_\e \otimes u_\e ) &= \frac{1}{\epsilon} \Delta(\rho_\e u_\e) - \frac{2}{\epsilon} \dive (\rho_\e Du_\e) + \frac{1}{\epsilon} \nabla \dive(\rho_\e u_\e), 
\\
\frac{1}{\epsilon} \dive (\rho_\e \nabla \log \rho_\e \otimes \nabla \log \rho_\e) & = \frac{1}{\epsilon} \Delta(\rho_\e \nabla \log \rho_\e) - \frac{1}{\epsilon} \dive (\rho_\e \nabla^2 \log \rho_\e).
\end{align*}
we obtain the  following alternative version of $\eqref{qns}$:
\begin{equation}\label{newqns}
\begin{aligned}
&\partial_t \rho_\e + \frac{1}{\epsilon} \dive (\rho_\e w_\e) = \frac{1}{\epsilon} \Delta \rho_\e \\
&\partial_t(\rho_\e w_\e) + \frac{1}{\epsilon} \dive (\rho_\e w_\e \otimes w_\e) + \frac{1}{\epsilon} \nabla \rho_\e^{\gamma} + \frac{1}{\epsilon} \dive (\rho_\e Dw_\e) - \frac{1}{\epsilon} \Delta(\rho_\e w_\e) \\
&\ - \frac{1}{\epsilon} \dive (\rho_\e \nabla^2 \log \rho_\e) 
 + \frac{1}{\e} \re \nabla V_{\e}+ \frac{1}{\epsilon^2} \rho_\e u_\e = 0.
\end{aligned}
\end{equation}

Now we pass to compute the energy associated the the system $\eqref{newqns}$. To this end,  as already done for $\eqref{qns}$, we use the continuity equation $\eqref{newqns}_1$ and multiply the equation  $\eqref{newqns}_2$ by $w_\e$ to conclude
\begin{equation*}
\begin{split}
&\partial_t \left(\frac{1}{2} \rho |w_\e|^2 \right) + \frac{1}{\epsilon}\dive \left(\rho_\e w_\e \frac{1}{2}|w_\e|^2\right) + \frac{1}{\epsilon}\left(\frac{1}{2} \Delta \rho_\e |w_\e|^2 - \Delta(\rho_\e w_\e)\cdot w_\e \right) + \\
&\frac{1}{\epsilon} \nabla \rho_\e^{\gamma} \cdot w_\e + \frac{1}{\epsilon} \dive(\rho_\e Dw_\e) \cdot w_\e - \frac{1}{\epsilon} \dive(\rho_\e \nabla^2 \log \rho_\e) \cdot w_\e + \frac{1}{\epsilon^2} \rho_\e u_\e w_\e + \frac{1}{\e}\re \nabla V_\e w_\e= 0.
\end{split}
\end{equation*}
Since
\begin{equation*}
\frac{1}{\epsilon}\left(\frac{1}{2} \Delta \rho_\e |w_\e|^2 - \Delta(\rho_\e w_\e)\cdot w_\e \right) = \frac{1}{\epsilon}  \left( \dive\left ( \nabla \rho_\e \frac{1}{2}|w_\e|^2 \right ) - \dive (\rho_\e \cdot \nabla w_\e) \cdot w_\e \right),
\end{equation*}
 when we integrate the equality above over $\T$ we get:
\begin{align}\label{bd2}
&\frac{d}{dt} \int_{\T} \frac{1}{2} \rho_\e |w_\e|^2 dx  + \frac{1}{\epsilon} \int_{\T} \rho_\e |\nabla w_\e|^2 dx  + \frac{1}{\epsilon} \int_{\T} \nabla \rho_\e^{\gamma} \cdot w_\e  dx + \frac{1}{\e}\int_{\T} \re \nabla V_{\e} w_\e \;dx  \nonumber\\
& - \frac{1}{\epsilon} \int_{\T} \rho_\e |D(w_\e)|^2dx + \frac{1}{\epsilon^2} \int_{\T} \rho_\e u_\e w_\e dx + \frac{1}{\epsilon} \int_{\T} \rho_\e \nabla^2 \log \rho_\e : \nabla w_\e \;dx  = 0.	
\end{align}
We observe that $|\nabla w_\e|^2- |D(w_\e)|^2 = |A(w_\e)|^2=|A(u_\e)|^2$, and, from the definition of $w_\e$, we infer
\begin{align*}
 \nabla \rho_\e^{\gamma} \cdot w_\e 
 &= \nabla \rho_\e^{\gamma} \cdot u_\e  + \frac{4}{\gamma} |\nabla \rho_\e^{\gamma/2}|^2 ;
 \\
 \rho_\e u_\e w_\e &= \rho_\e u_\e^2 + \rho_\e u_\e \nabla \log \rho_\e ;
 \\
\rho_\e \nabla^2\log \rho_\e: \nabla w_\e &= \rho_\e \nabla^2\log \rho_\e : \nabla u_\e + \rho_\e |\nabla^2\log \rho_\e|^2\\
\re \nabla V_{\e} w_\e &= \re \nabla V_{\e} \ue + \re \nabla V_{\e} \nabla \log \re
\end{align*}
Moreover, we recall \eqref{press}, \eqref{effqns} and \eqref{potential} to compute:
\begin{equation*}
    \frac{d}{dt} \int_{\mathbb{T}^3} \big (\rho_\e(\log \rho_\e - 1) + 1 \big )dx = \frac{1}{\epsilon} \int_{\T} \rho_\e u_\e \nabla \log \rho_\e \;dx ,
\end{equation*}
\begin{equation*}
\frac{1}{\e} \int_{\T} \re \nabla V_{\e} w_\e \;dx = \frac{d}{dt} \int_{\T} \frac{1}{2} |\nabla V_{\e}|^2 \;dx  + \frac{1}{\e}\int_{\T} \re \nabla V_{\e} \nabla \log \re \;dx ,
\end{equation*}
the last term can be rewrite:
\begin{equation*}
\frac{1}{\e}\int_{\T} \re \nabla V_{\e} \nabla \log \re \;dx = \frac{1}{\e} \int_{\T} \nabla V_{\e} \nabla \re\;dx  = -\frac{1}{\e}\int_{\T} \Delta V_{\e} \re \;dx = \frac{1}{\e} \int_{\T} \re(\re -g_\e) \;dx   
\end{equation*}
Therefore we multiply $\eqref{bd2}$ by $\epsilon$ and, using the previous identities, we obtain the final estimate
\begin{align*}
&\frac{d}{dt} \int_{\mathbb{T}^3} \left [\epsilon \left(\frac{1}{2} \rho |w_\e|^2 + \frac{\rho_\e^{\gamma}}{\gamma -1} + 2|\nabla \sqrt{\rho_\e}|^2 + \frac{1}{2}|\nabla V_{\e}|^2 \right) + \rho_\e( \log \rho_\e - 1) +1 \right ]dx \\
& +\int_{\mathbb{T}^3} \rho |A(u_\e)|^2 \;dx + \frac{4}{\gamma} \int_{\mathbb{T}^3} |\nabla \rho_\e^{\frac{\gamma}{2}}|^2 \;dx  + \int_{\mathbb{T}^3} \rho_\e |\nabla^2 \log \rho_\e|^2 \;dx  + \frac{1}{\epsilon} \int_{\mathbb{T}^3} \rho_\e u_\e^2 \;dx + \int_{\T} \re(\re - g_\e) \;dx = 0
\end{align*}
which gives $\eqref{newbdestimate}$ upon time integration. 
\end{proof}

\subsection{Weak solution of the Quantum-Drift-Diffusion equation}
Next, we consider the Quantum Drift--Diffusion--Poisson equation in $(0,T)\times\T$ with $g:\T \rightarrow \R$
\begin{equation}\label{eq:qdd}
\begin{aligned}
    &\d_t\rho + \dive \left(2\rho \nabla\left( \frac{\Delta \sqrt{ \rho}}{\sqrt{\rho}}\right) - \nabla \rho^{\gamma} - \rho \nabla V \right) = 0 \\
    & - \Delta V = \rho - g,
\end{aligned}
\end{equation}
with initial datum
\begin{equation}\label{eq:qddid}
\rho|_{t=0}=\rho^0\geq 0,
 %\textcolor{blue}{V\vert_{t=0} = V^0 \text{ such that } -\Delta V^0 = \rho^0 -g}
\end{equation}
and $V$ such that
\begin{equation}\label{meanVQDD}
    \fint_{\T} V dx = 0.
\end{equation}
As before, to solve    \eqref{eq:qdd}$_2$, \eqref{meanVQDD} we need the following assumption  for   $g$:  defining $M:=\fint_{\T} \rho^0$, we assume that $\fint_{\T} g = M$.

The following definition  specifies the framework of weak solutions to  \eqref{eq:qdd}  we shall obtain at the limit.
\begin{definition}\label{WSGF}
Given $g \in L^2(\T)$, we say that $(\rho, V)$	with $\rho \geq 0$ is a finite energy weak solution of \eqref{eq:qdd} if
	\begin{itemize}
	    \item[(1)] Integrability condition: \\
	    \begin{equation*}
	    \begin{aligned}
	        &\rrho\in L^{\infty}(0,T;H^{1}(\T))\cap L^{2}(0,T;H^2(\T));\quad 
	        &\rho\in L^{\infty}(0,T;L^{\gamma}(\T))\\
	        & V \in C((0,T);H^2(\T))	        
	    \end{aligned}
	\end{equation*}
	\item[(2)] Continuity equation: 
	\begin{equation*}
	\int_{\T}\rho^0 \phi(0) \;dx + \iint \big (\rho \phi_t-\rho^{\gamma} \dive \nabla \phi-2\nabla \sqrt{\rho}\otimes\nabla\sqrt{\rho}:\nabla^2 \phi+2\sqrt{\rho} \nabla^2 \sqrt{\rho}:\nabla^2 \phi + \rho \nabla V \cdot \phi \big) \;ds dx = 0,
	\end{equation*} 
	for any  $\phi \in C_c^{\infty} ([0,T);C^{\infty} (\mathbb{T}^3))$,
	\item[(3)] Poisson equation:\\
		For a.e. $(t,x)\in (0,T)\times\T$ it holds that 
		\begin{equation}\label{eq:poissonQDD}
		-\Delta V=\rho-g
		\end{equation}
	\item [(4)] Entropy inequalities:\\
	 there exist $\Lambda,\,\mathcal{S} \in L^{2}((0,T)\times\T)$ such that 
\begin{align}
&\rrho\Lambda=\dive (2\rrho\nabla^2\rrho-2\nabla\rrho\otimes\nabla\rrho-\rho^{\gamma}\mathbb{I})- \rho \nabla V \textrm{ in }\mathcal{D}'((0,T)\times\T)\label{eq:idenlambda}\\
&\rrho\mathcal{S}=2\rrho\nabla^2\rrho-2\nabla\rrho\otimes\nabla\rrho\,\mbox{ a.e. in }(0,T)\times\T\label{eq:idens}
\end{align} 
and constant $C>0$ such that for a.e. $t\in(0,T)$ 
\begin{equation}\label{eq:form11}
\int_{\T}\left(|\nabla\rrho|^2+\frac{\rho^{\gamma}}{\gamma-1} + \frac{1}{2}|\nabla V|^2 \right)(t) \;dx + \int_0^t \int_{\T}|\Lambda|^2 \; ds dx  \leq C,
\end{equation} 
\begin{equation}\label{eq:form22}
\begin{aligned}
&\int_{\T}(\rho(\log\rho-1)+1)(t) \;dx +\int_0^t\int_{\T}|\mathcal{S}|^2 \;ds dx +\frac{4}{\gamma}\int_0^t\int_{\T}|\nabla\rho^{\frac{\gamma}{2}}|^2\;ds dx + \int_0^t\int_{\T} \rho(\rho - g)\;ds dx   \\
& \leq \int_{\T} (\rho^0 (\log\rho^0 -1)+1) \;dx .
\end{aligned}
 \end{equation}
 \end{itemize}
In addition,  $(\rho, V)$ is called \emph{energy dissipating weak solution} if in particular it holds that for a.e. $t\in (0,T)$: 

\begin{equation}\label{eq:form23}
\begin{aligned}
\int_{\T}\left(|\nabla\rrho|^2+\frac{\rho^{\gamma}}{\gamma-1}+ \frac{1}{2}|\nabla V|^2 \right)(t)\,dx+\int_0^t\int_{\T}|\Lambda|^2\,dsdx &\leq \int_{\T}\,|\nabla\sqrt{\rho^0}|^2\,dx+\int_{\T}\frac{{\rho}^{0 \gamma}}{\gamma-1}\,dx\\
&+\int_{\T}\frac{1}{2} |\nabla V(0,x)|^2.
\end{aligned}
\end{equation} 
\end{definition}
\begin{remark}
The terminology finite energy and energy dissipating weak solutions is motivated by the fact that both \eqref{eq:form11} and \eqref{eq:form23} arises as a weak limit of the energy inequality \eqref{fes-1}. As in Proposition \ref{prop:energybd} we derive the formal estimates in the framework of smooth solutions.
\end{remark}

%	\textcolor{red}{The presence of $\mathcal{S}$ and $\Lambda$ in the above definition is justified in view of Lemma \ref{lem:conv} below, where  we define the tensor $\mathcal{S}$ as the weak limit of $\mathcal{S_\e}$ and $\Lambda$ as the weak limit of $\frac{1}{\epsilon} \sqrt{\rho_\e}u_\e$, both in $L^2((0,T);\T)$. As already pointed out in Remark \ref{rem:1}, $\mathcal{S_\e}$ emerges in the regularization procedure for $\sqrt{\rho_\e} \nabla^2 \log \rho_\e$ in the construction  of weak solutions, while in Theorem \ref{teo:weakweak} we prove that $\sqrt{\rho} \Lambda = \dive (-2 \nabla \sqrt{\rho} \otimes \sqrt{\rho} + 2 \sqrt{\rho}  \nabla^2 \sqrt{\rho} - \rho^{\gamma} \mathbb{I}) $ in distributional sense, thus verifying the weak formulation of the  continuity equation \eqref{eq:qdd}.}
%\\
\begin{proposition}\label{enebdqddreg}
Let    $(\rho,V)$ be a smooth solution of \eqref{eq:qdd} with data satisfying \eqref{eq:qddid},\eqref{meanVQDD}. Then, for any $t\in[0,T)$ the pair $(\rho, V)$ satisfies the following estimates:  
\begin{align}
	&\frac{d}{dt}\int_{\T} \left (2|\nabla\rrho|^2+\frac{\rho^\gamma}{\gamma-1} + \frac{1}{2}|\nabla V|^2 \right )dx +\int_{\T}\rho\left|2\nabla\left(\frac{\Delta\rrho}{\rrho}\right)-\gamma\rho^{\gamma-2}\nabla\rho- \nabla V \right|^2 \;dx  =0,\label{eq:form1}\\
	&\frac{d}{dt}\int_{\T} \rho(\log\rho-1)+1+ \;dx \int_{\T} \rho|\nabla^2\log\rho|^2\;dx +\frac{4}{\gamma}\int_{\T}|\nabla\rho^{\frac{\gamma}{2}}|^2\;dx  + \int_{\T} \rho(\rho -g)\;dx =0.\label{eq:form2}
	\end{align}
\end{proposition}

\begin{proof}
The energy \eqref{eq:form2} is achieved by multiplying \eqref{eq:qdd} by $\log\rho$ and integrating  by parts:
	\begin{align*}
	& \int_{\T}\rho_t \log \rho \;dx  = \frac{d}{dt} \int_{\T}  \big ((\rho (\log \rho - 1)) +1\big ) \;dx, \\
	& \int_{\T} \dive \left( 2\rho \nabla\left( \frac{\Delta \sqrt{ \rho}}{\sqrt{\rho}}\right) - \nabla \rho^{\gamma} \right) \log \rho \;dx  = \int_{\T}\dive \left( \dive\left(\rho \nabla^2 \log \rho\right) - \nabla \rho^{\gamma} \right)\log \rho \;dx \\
	& = \int_{\T} \rho|\nabla^2\log\rho|^2 \;dx +\frac{4}{\gamma}\int_{\T}|\nabla\rho^{\frac{\gamma}{2}}|^2 \;dx , \\
	&- \int_{\T}  \dive(\rho \nabla V) \log \rho \;dx = \int_{\T}\rho \nabla V \nabla \log \rho \;dx = - \int_{\T} \Delta V \rho\;dx = \int_{\T} \rho(\rho - g)\;dx .
	\end{align*}
	
	Moreover, if we multiply \eqref{eq:qdd}  by $- 2\Delta\rrho/\rrho+\gamma\rho^{\gamma-1}/(\gamma-1)$, after integrating by parts we get
	\begin{align*}
	& \int_{\T} \rho_t \left( \frac{-2\Delta \sqrt{\rho}}{\sqrt{\rho}} \right) \;dx  = -\int_{\T} \dive \left(  2\frac{\nabla \sqrt{\rho}}{\sqrt{\rho}} \rho_t \right) \;dx  + \int_{\T} \nabla \sqrt{\rho} \nabla \left(2 \frac{\rho_t}{\sqrt{\rho}} \right)\;dx  =  \frac{d}{dt}\int_{\T} 2 |\nabla \sqrt{\rho}|^2 \;dx ,
	\\
	&\int_{\T} \rho_t \left( \frac{\gamma}{\gamma -1} \rho^{\gamma -1} \right) \;dx  = \frac{d}{dt}\int_{\T} \frac{\rho^{\gamma}}{\gamma -1 } \;dx ,
	\\
	& \int_{\T} \dive \left(2\rho \nabla \left( \frac{\Delta \sqrt{\rho}}{\sqrt{\rho}} \right)\right) \left(- 2 \frac{\Delta \sqrt{\rho}}{\sqrt{\rho}} + \frac{\gamma}{\gamma -1} \rho^{\gamma -1} \right)\;dx  =\\
	& = \int_{\T}\left ( 4\rho \left| \nabla \left(  \frac{\Delta \sqrt{\rho}}{\sqrt{\rho}} \right) \right|^2 - 2\gamma \rho \nabla \left(  \frac{\Delta \sqrt{\rho}}{\sqrt{\rho}} \right)\rho^{\gamma-2} \nabla \rho  \right ) \;dx ,
	\\
	&\int_{\T} \dive(-\nabla \rho^{\gamma}) \left(- 2 \frac{\Delta \sqrt{\rho}}{\sqrt{\rho}} + \frac{\gamma}{\gamma -1} \rho^{\gamma -1} \right) \;dx =  \int_{\T} -\nabla \rho^{\gamma} \nabla \left( 2\frac{\Delta \sqrt{\rho}}{\sqrt{\rho}} \right) + \rho|\gamma \rho^{\gamma -2}\nabla \rho|^2 \;dx \\
	& = \int_{\T} -2\gamma \rho \rho^{\gamma -2}\nabla \rho \nabla \left(  \frac{\Delta \sqrt{\rho}}{\sqrt{\rho}} \right) + \rho|\gamma \rho^{\gamma -2}\nabla \rho|^2 \;dx ,\\
	& \int_{\T} - \dive(\rho \nabla V) \frac{\gamma}{\gamma -1}\rho^{\gamma -1}\;dx  = \int_{\T}\rho \nabla V \gamma \rho^{\gamma-2} \nabla \rho \;dx  = \int_{\T} \nabla V \nabla \rho^{\gamma}\;dx  = -\int_{\T} \Delta V \rho^{\gamma} \;dx  \\
	& \int_{\T} \dive(\rho \nabla V) 2\frac{\Delta \sqrt{\rho}}{\sqrt{\rho}}\; dx  = - \int_{\T} \rho \nabla V 2\nabla \left(\frac{\Delta \sqrt{\rho}}{\sqrt{\rho}}\right) \;dx = - \int_{\T} \nabla V \dive(\rho \nabla^2 \log \rho) \; dx 
	\end{align*}
	Finally we multiply \eqref{eq:qdd} by $V$, integrating by parts we get:
	\begin{equation*}
	    \int_{\T} \rho_t V\;dx  = - \int_{\T} \Delta V_t\;V \;dx  = \int_{\T} \nabla V_t \nabla V \;dx  = \frac{1}{2}\frac{d}{dt} \int_{\T} |\nabla V|^2 \;dx ,
	\end{equation*}
\begin{equation*}
    \int_{\T} \dive \left(2 \rho \nabla \left(\frac{\Delta \sqrt{\rho}}{\sqrt{\rho}}\right) \right) V \;dx = - \int_{\T}2 \rho \nabla \left(\frac{\Delta \sqrt{\rho}}{\sqrt{\rho}}\right) \nabla V\;dx  =  - \int_{\T} \nabla V \dive(\rho \nabla^2 \log \rho)\;dx ,
\end{equation*}
\begin{equation*}
    \int_{\T} -\dive(\nabla \rho^{\gamma}) V =  \int_{\T}\nabla \rho^{\gamma} \nabla V = -  \int_{\T} \rho^{\gamma} \Delta V 
\end{equation*}
and \eqref{eq:form1} follows by summing up all terms.
\end{proof}

\begin{remark}
	It is worth to observe that if we perform the Hilbert expansion of $\eqref{qns}$ the  limit solution $(\bar \rho, \bar u)$ formally satisfies at the first non trivial order $O(1/\e)$ (the order $O(1/\e^2)$ tells us the momentum expansion starts form  the power one in $\e$)  the following identities:
	\begin{align*}
	\bar \rho \bar u = \e \bar \rho \left( 2\nabla \left( \frac{\Delta \sqrt{\bar\rho}}{\sqrt{\bar\rho}}\right) - \gamma \bar\rho^{\gamma-2} \nabla \bar\rho - \nabla \bar V \right),
\\
	\partial_t \bar \rho + \dive \left( 2 \bar\rho\nabla \left( \frac{\Delta \sqrt{\bar\rho}}{\sqrt{\bar\rho}}\right) - \nabla \bar\rho^{\gamma} - \bar \rho \nabla \bar V \right)= 0,
	\end{align*}
	which is exactly the Quantum Drift--Diffusion equation $\eqref{eq:qdd}$. We underline also that, with the above definition  for $\bar u$, namely 
	\begin{equation*}
	    	(\bar \rho, \bar u)= \left(\bar \rho, \e \left( 2\nabla \left( \frac{\Delta \sqrt{\bar\rho}}{\sqrt{\bar\rho}}\right) - \gamma \bar\rho^{\gamma-2} \nabla \bar\rho - \nabla \bar V \right) \right),
	    		\end{equation*}
in the formulation for 
	the mechanical energy $\eqref{fes}$ for $\eqref{qns}$, at the limit the latter  reduces to  $\eqref{eq:form1}$ after a time integration, that is the corresponding  conservation of the mechanical energy for the Quantum Drift--Diffusion equation. The same happens for the BD entropy, namely $\eqref{newbdestimate1}$ for $\e\to 0$ reduces to (the  time integrated version of) \eqref{eq:form2}. This fact is coherent with our analysis and it will be validated by Theorem $\ref{teo:weakweak}$ below.
\end{remark}
\begin{remark}
We remark that if $\rho$ is a weak solution in the sense of the previous definition the initial datum $\rho^0$ is attained for example in the strong topology of $L^{2}(\T)$. 
\end{remark}

\section{Main result}\label{sec:3}
In this section we prove Theorem \ref{teo:mainintro}, which we rewrite for reader's convenience. 
\begin{theorem}\label{teo:weakweak}
Let $(\re,\ue, V_\e)$ be a weak solution of $\eqref{qns}$ in the sense Definition \ref{WSQNS} with data \eqref{eq:id} satisfying 
\begin{equation}\label{eq:idn}
\begin{aligned}
&\{\rho_\epsilon^0\}_{\e}\mbox{ is bounded in }L^1 \cap L^{\gamma} (\T)\ \hbox{such that}\ \rho_\epsilon^0 \to \rho^0\ \hbox{in}\ L^q(\T),\, q<3 \\ 
&\{\nabla \sqrt{\rho_\epsilon^0}\}_{\e}\mbox{ is bounded in }L^2(\T),\\
&\{\sqrt{\rho_\epsilon^0 }u_\epsilon^0\}_{\e}\mbox{ is bounded in }L^2(\T),\\
&\{ g_\e \}_{\e} \in L^2(\T) \mbox{ and  } g_\e \weakto g \mbox{ weakly in }  L^2(\T).
\end{aligned}
\end{equation}
Then, up to subsequences, there exist $\rho\geq 0$ and $V$ such that 
\begin{align*}
& \rre\rightarrow \rrho\textrm{ strongly in }L^{2}((0,T);H^{1}(\T)) \\
& \nabla V_\e \rightarrow \nabla V \text{ strongly in } C([0,T);L^{2}(\T))
\end{align*}
and $(\rho, V)$ is a finite energy weak solution of $\eqref{eq:qdd}-\eqref{eq:qddid}$ in the sense of  Definition \ref{WSGF}.
If in addition to \eqref{eq:idn}, it also hold that 
\begin{equation}\label{eq:convid}
\begin{aligned}
&\sqrt{\rho_\e^0}\ue^0\to 0\mbox{ strongly in }L^{2}(\T)\\
&\nabla \sqrt{\rho_\e^0} \to \nabla\sqrt{\rho^0}\mbox{ strongly in }L^{2}(\T)\\
&\rho_\e^0 \to\rho^0\mbox{ strongly in } L^{\gamma}(\T)\\
\end{aligned}
\end{equation}
then $(\rho,V)$ is an energy dissipating   weak solution.
\end{theorem}
The proof of Theorem \ref{teo:weakweak} requires several preliminaries, collected in the following section.
\subsection{Preliminary results} We start by proving the uniform bounds obtained by the Energy estimate and BD Entropy. 
\begin{lemma}\label{lem:unifaprioriinitial}
There exists a constant $C>0$ such that the following uniform bounds on the data 
\begin{align}\label{uniformbounds}
&\sup_t \int_{\T} \rho_{\epsilon} \;dx \leq C;\quad\sup_t \int_{\T} \rho_{\epsilon} u_\epsilon^2 \;dx  \leq C;\quad \int_0^T \int_{\T}|\nabla\re^{\frac{\gamma}{2}}|^2 \; dsdx \leq C;\quad\sup_t \int|\nabla\sqrt{\rho_{\epsilon}}|^2\;dsdx  \leq C;\nonumber\\
&\frac{1}{\epsilon} \int_0^T \int_{\T} |\MT_{\e}^{s}|^2 \;dsdx  \leq C;\quad\sup_t \int \rho_{\epsilon}^{\gamma}\;dx  \leq C;\quad \frac{1}{\e^2}\int_0^T\int\rho_{\epsilon} u_\epsilon^2 \;dsdx  \leq C; \nonumber\\
& \int_0^T \int_{\T} |\nabla^2 \sqrt{\rho_{\epsilon}}|^2 \;dsdx \leq C; \quad  \int_0^T \int_{\T} |\mathcal{S}_{\epsilon}|^2 \;dsdx  \leq C; \quad \sup_t \int_{\T} |\nabla^2 V_{\e}|^2 \;dx  \leq C
\end{align}
\end{lemma}
\begin{proof}
We first notice that under the hypothesis \eqref{eq:idn} we have that 
\begin{equation}\label{eq:log1}
\{\re^0\log\re^0\}_{\e}\text{ is bounded in }L^{1}(\T). 
\end{equation}
This is obtained very easily under the hypothesis \eqref{eq:idn}. Indeed, one can easily show that for $q>p\geq1$ there exists $C>0$ such that for $s\geq 0$:
\begin{equation*}
|s\log s|^q\leq C(1+s^p)
\end{equation*}
and therefore, by taking $q=\gamma$ and $1\leq p<\gamma$ we have that
\begin{equation}\label{eq:log}
\{\re^0\log\re^0\}_{\e}\text{ is bounded in }L^{q}(\T). 
\end{equation}
which implies \eqref{eq:log1}.

Next, as shown in Remark \ref{rem:4},  $\{\nabla V_{\e}|_{t=0}\}_{\e}$ is bounded in $L^{2}(\T)$ and  we have proved that the right-hand sides of \eqref{fes-1} and \eqref{newbdestimate1} are bounded uniformly in $\e$. Then, to obtain \eqref{uniformbounds} is enough to use the bound
\begin{equation*}
\int_0^t\int_{\T} \re(\re-g_{\e}) \;dsdx \geq \frac{1}{2} \int_0^t\int_{\T} |\re|^2 \;dsdx -\frac{1}{2} \int_0^t\int_{\T} |g_{\e}|^2 \;dsdx
\end{equation*}
in the energy inequalities \eqref{fes-1} and \eqref{newbdestimate1}.
Finally, the bound on the potential follows from the Poisson equation. 
\end{proof}

The following lemma we prove the convergence needed to pass to the limit as $\e\to0$. 
\begin{lemma}\label{lem:rho}
There exist $\rho\geq 0$ such that $\rrho\in L^{\infty}(0,T);H^{1}(\T)\cap L^{2}(0,T;H^{2}(\T))$ and the following hold:
\begin{align}
	&\rho_{\epsilon} \rightarrow \rho\text{ in $C([0,T];L^{q}(\T)$ },\, q<3,\label{eq:rho1}\\
	&\rrho_{\epsilon} \weakto \rrho\text{ weakly in $L^2(0,T; H^2 (\mathbb{T}^3))$},\label{eq:rho2}\\
		&\rre\weaktos\rrho\text{ weakly* in $L^\infty(0,T; H^1 (\mathbb{T}^3))$},\label{eq:rho6}\\
	&\sqrt{\rho_{\epsilon}} \rightarrow \sqrt{\rho}\text{ strongly in }L^2(0,T;H^1(\mathbb{T}^n))\label{eq:rho5}.\\
	&\re(\log\re+1)+1\to \rho(\log\rho+1)+1\text{ strongly in }L^{1}(0,T;L^{1}(\T)),\label{eq:rhologrho}\\
	&\rho_\epsilon^{\gamma} \rightarrow \rho^{\gamma}\text{ in $L^1(0,T;L^1(\mathbb{T}^n))$}\label{eq:rho4}\\
%	& \nabla V_{\e} \weaktos \nabla V \text{ weakly* in } L^{\infty}((0,T);L^2(\T)) \label{eq:V}
	\end{align}
\end{lemma}
\begin{proof}
From $\eqref{uniformbounds}$ and Sobolev embedding have that:
	\begin{equation}\label{eq:sob}
	\{\sqrt{\rho_{\epsilon}}\}_{\e}\text{ is bounded in } L^2(0,T;L^{q}(\T))\quad q\in[1,6]
	\end{equation}
	Therefore since $\nabla \rho_{\epsilon} =2\nabla \sqrt{\rho_{\epsilon}} \sqrt{\rho_{\epsilon}}$ 
	by using \eqref{eq:sob} and \eqref{uniformbounds} we can infer that 
	\begin{equation*}
	\{\rho_{\epsilon}\}_{\e}\mbox{ is bounded in }L^{\infty}(0,T;W^{1,\frac{3}{2}}(\T)).
	\end{equation*}
	Next, by using the weak formulation of the continuity equation and the bounds \eqref{uniformbounds} we get that 
	\begin{equation*}
	\{\partial_t\re\}_{\e}\mbox{ is bounded in }L^{2}(0,T;W^{-1,\frac{3}{2}}(\T)).
	\end{equation*}
	Indeed, it is enough to note that we have
	\begin{equation*}
	\int_0^T\|\partial_t\re\|^2_{W^{-1,\frac{3}{2}}} \;ds
	\leq  \frac{1}{\epsilon^2}\int_0^T\|\rho_\epsilon u_\epsilon\|^2_{L^\frac{3}{2}_x} \;ds
	\leq \int_0^T\|\rre\|_{L_x^{6}}^2\left\|\rre\frac{\ue}{\e}\right\|^{2}_{L^{2}_x}\;ds.
	\end{equation*}
	Since for $q<3$ we have $W^{1,3/2}(\T) \subset L^{q}(\T)$ with compact embedding and $L^{q}(\T)\subset W^{-1,3/2}(\T)$ 
	we can apply Aubin-Lion lemma to deduce that there exists a subsequence not relabelled and $\rho\geq 0$ such that \eqref{eq:rho1} holds. Moreover, by passing to a further
	subsequence if necessary, \eqref{eq:rho1} implies 
	\begin{equation}\label{eq:aerho}  
	\rho_\epsilon \rightarrow \rho \; \; \text{ a.e in }(0,T)\times\T. 
	\end{equation}
	Then, the convergence \eqref{eq:rho2} and \eqref{eq:rho4} follow from \eqref{eq:aerho}, the uniform bounds \eqref{uniformbounds} and standard weak compactness considerations.
	Next, to prove \eqref{eq:rho5} we start by proving that 
	\begin{equation}\label{eq:strongrrho2}
	\rre\to\rrho\text{ strongly in }L^{2}(0,T;L^{2}(\T)). 
	\end{equation}
	First, notice that  \eqref{eq:aerho} implies that
	\begin{equation*}\sqrt{\rho_{\epsilon}} \rightarrow \sqrt{\rho} \; \; \text{ a.e in (0,T) $\times \mathbb{T}^3$}.\end{equation*}
	Let $M>0$, then
\begin{equation*}
\begin{split}
\int_0^T \int_{\mathbb{T}^3}|\sqrt{\rho_{\epsilon}}  - \sqrt{\rho}|^2 \;dsdx \leq & \int_0^T \int_{\{\rho_{\epsilon}> M\}} |\sqrt{\rho_{\epsilon}}|^2 \;dsdx + \\
&\int_0^T \int_{\mathbb{T}^3}  |\sqrt{\rho_{\epsilon}} \chi_{\{\rho_{\epsilon}\leq M\}} - \sqrt{\rho} \chi_{\{\rho \leq M\}}|^2 \;dsdx \; + \\ & \int_0^T \int_{\{\rho> M\}} |\sqrt{\rho}|^2 \;dsdx.
\end{split}
\end{equation*}
Then, by using \eqref{uniformbounds}, \eqref{eq:sob}, the fact that $\rrho\in L^{\infty}(0,T;L^{6}(\T))$, 
we have that
\begin{equation*}
\begin{split}
\int_0^T \int_{\mathbb{T}^3}|\sqrt{\rho_{\epsilon}}  - \sqrt{\rho}|^2 \;dsdx \leq &\frac{1}{M^2} \int_0^T \int |\sqrt{\rho_{\epsilon}}|^6 \;dsdx \\
&+\int_0^T \int_{\mathbb{T}^3}  |\sqrt{\rho_{\epsilon}} \chi_{\{\rho_{\epsilon} \leq M\}} - \sqrt{\rho} \chi_{\{\rho\leq M\}}|^2 \;dsdx \;\\ & +\frac{1}{M^2}\int_0^T \int |\sqrt{\rho}|^6 \;dsdx\\
&\leq\int_0^T \int_{\mathbb{T}^3}  |\sqrt{\rho_{\epsilon}} \chi_{\{\sqrt{\rho_{\epsilon}} \leq M\}} - \sqrt{\rho} \chi_{\{\sqrt{\rho} \leq M\}}|^2\;dsdx+\frac{2C}{M^2}
\end{split}
\end{equation*}
Then, we conclude by first sending $\e\to 0$ in the second term, where we use Dominated Convergence Theorem, and then by choosing suitably $M\to\infty$.

The strong convergence \eqref{eq:rho5} of $\rre$ in $L^{2}(0,T;H^{1}(\T))$ is a consequence of the following simple interpolation inequality: 
%. Indeed,
\begin{equation}\label{interpol}
\|\rre(t)-\rrho(t)\|_{H^1}\leq C\|\rre(t)-\rrho(t)\|_{L^2}^{\frac{1}{2}}\|\rre(t)-\rrho(t)\|_{H^{2}}^{\frac{1}{2}}.
\end{equation}

%\hole{This is very easy to prove both with Fourier series or with a trivial integration by parts argument. Indeed, let $f$ periodic, then
%\begin{equation*}
%\begin{aligned}
%\int|\nabla f|^{2}&=\int\nabla f\cdot\nabla f
                              %=-\int f\Delta f\\
                              %&\leq C\int|f||\nabla^2 f|
                  %           \leq C\left(\int| f|^2\right)^{\frac{1}{2}}\left(\int |\nabla^2 f|^2\right)^{\frac{1}{2}}
%\end{aligned}
%\end{equation*}}
%and then integrating in time, using H\"older inequality, \eqref{uniformbounds} and then \eqref{eq:strongrrho2}
%we get \eqref{eq:rho5}. 
Next, we prove \eqref{eq:rhologrho}. We first notice that by using \eqref{uniformbounds} and by the very same argument used to deduce \eqref{eq:log}, we easily have that for some $p>1$
\begin{equation*}
\begin{aligned}
&\rho\log\rho\in L^{p}((0,T)\times\T),\\
&\{\re\log\re\}_{\e}\text{ is bounded in }L^{p}((0,T)\times\T). 
\end{aligned}
\end{equation*}
Moreover, since the function $s\to s\log\,s$ is continuous on $[0,\infty)$ we have that 
\begin{equation*}
\re\log\re\to\rho\log\rho\text{ a.e. in }(0,T)\times\T. 
\end{equation*}
Let $M>0$, then
\begin{equation*}
\begin{split}
\int_0^T \int_{\mathbb{T}^3}|\re\log\re-\rho\log\rho| \;dsdx \leq & \int_0^T \int_{\{\rho_{\epsilon}> M\}} |\re\log\re| \;dsdx\\
&+\int_0^T \int_{\mathbb{T}^3}  |\re\log\re\chi_{\{\rho_{\epsilon} \leq M\}} - \rho\log\rho \chi_{\{\rho \leq M\}}| \;dsdx \; \\ &+ \int_0^T \int_{\{\rho> M\}} |\rho\log\rho|\;dsdx.\\
&\leq\int_0^T \int_{\mathbb{T}^3}  |\re\log\re\chi_{\{\rho_{\epsilon} \leq M\}} - \rho\log\rho \chi_{\{\rho \leq M\}}|\;dsdx \; \\ 
&+\frac{2C}{(M\log M)^{p-1}}
\end{split}
\end{equation*}
and we conclude as in before. 
Finally, we prove the convergence of the pressure term. We first note that  from \eqref{uniformbounds} we have that 
	\begin{equation*}
	\{\re^{\frac{\gamma}{2}}\}_{\e}\mbox{ is bounded in }L^{\infty}(0,T;L^{2}(\T))\cap L^{2}(0,T;H^{1}(\T)).
	\end{equation*}
	Then, by Sobolev embedding
	$$
	\{\rho_\e^{\frac{\gamma}{2}}\}_{\e}\text{ is bounded in } L^2(0,T;L^{6}(\T))
	$$ 
	and therefore  for a.e $t$:
	\begin{equation*}
	\|\rho_\e^{\frac{\gamma}{2}}(t)\|_{\frac{10}{3}} \leq \|\rho_\e^{\frac{\gamma}{2}}(t)\|^{\frac{2}{5}}_{2}\| \rho_\e^{\frac{\gamma}{2}}(t)\|_{6}^{\frac{3}{5}}.
	\end{equation*}
Therefore, by integrating in time and using \eqref{uniformbounds} we have that
	\begin{equation*}
	\{\re^{\frac{\gamma}{2}}\}_{\e}\mbox{ is bounded in }L^{\frac{10}{3}}((0,T)\times\T),
	\end{equation*}
	which is equivalent to say that 
	\begin{equation*}
	\{\re^{\gamma}\}_{\e}\mbox{ is bounded in }L^{\frac{5}{3}}((0,T)\times\T).
	\end{equation*}
	Moreover, by \eqref{eq:aerho} we have that also $\re^{\gamma}\rightarrow\rho^{\gamma}$ a.e. in $(0,T)\times\T$ and by Fatou Lemma we have that 
	\begin{equation*}
	\rho^{\gamma}\in L^{\frac{5}{3}}((0,T)\times\T).
	\end{equation*}
	Then, if $M>0$ we have 
\begin{equation*}
\begin{split}
\int_0^T \int_{\mathbb{T}^3}|\re^{\gamma}  - \rho^{\gamma}| \;dsdx \leq & \int_0^T \int_{\{\rho_{\epsilon}> M\}} \re^\gamma \;dsdx\\
+&\int_0^T \int_{\mathbb{T}^3}  |\re^{\gamma}\chi_{\{\rho_{\epsilon}\leq M\}} -\rho^\gamma \chi_{\{\rho \leq M\}}| \;dsdx \; \\ +& \int_0^T \int_{\{\rho> M\}} \rho^{\gamma} \;dsdx\\
\leq &\int_0^T \int_{\mathbb{T}^3}  |\re^{\gamma}\chi_{\{\rho_{\epsilon}\leq M\}} -\rho^\gamma \chi_{\{\rho \leq M\}}|\;dsdx\\
+&\frac{2C}{M^{\frac{2}{3}\gamma}}
\end{split}
\end{equation*}
and we conclude as before. \\

%To prove \eqref{eq:V} we observe that from \eqref{eq:rho1} $\re \in C((0,T);L^2(\T))$  and $g_\e \in L^2(\T)$ for every $t \in [0,T]$, thus $V_{\e}$ is uniformly bounded in $L^{\infty}((0,T);H^2(\T))$. Even if $L^{\infty}((0,T);H^2(\T))$ is not reflexive we can characterize $L_t^{\infty}$ as $(L^1_t)^*$, namely the dual of $L_t^1$ that is a separable space. Thanks to Banach--–Alaoglu Theorem we claim there exists $f \in L^{\infty}((0,T);H^2(\T))$ such that $V_\e \weaktos f$ in $L^{\infty}((0,T);H^2(\T))$. This implies:
%\begin{equation*}
%    \begin{aligned}
%    &V_\e \weakto f \in L^{\infty}((0,T);L^2(\T))\\
%    &\nabla V_\e \weakto \nabla f \in L^{\infty}((0,T);L^2(\T))\\
%    & \Delta V_\e \weakto \Delta f \in L^{\infty}((0,T);L^2(\T)).
%    \end{aligned}
%\end{equation*}
%Using the Poisson equation
%\begin{equation*}
%    -\Delta V_{\e} = \re - g_\e,
%\end{equation*}
%and the fact that
%\begin{equation*}
%    \re - g_\e \weakto \rho - g = - \Delta V \text{ in } L^{\infty}((0,T);L^2(\T)),
%\end{equation*}
%one concludes by the uniqueness of the weak* limit that \begin{equation*}
%\Delta V_{\e}  \weaktos \Delta V \text{ weakly* in }  L^{\infty}((0,T); L^2(\T))
%\end{equation*}
%and also
%\begin{equation*}
%  \nabla V_{\e}  \weaktos \nabla V \text{ weakly* in }  L^{\infty}((0,T); L^2(\T)).
%\end{equation*}

\end{proof} 

\subsection{Proof of Theorem \ref{teo:weakweak}}
First, we notice that by \eqref{uniformbounds}, \eqref{eq:rho1} and the compact embedding of $L^{2}(\T)$ in $H^{-1}(\T)$ for $\{g_\e\}_\e$, we have that there exists $V\in C([0,T);H^{2}(\T))$ such that 
\begin{equation}\label{eq:convv}
\nabla V_{\e}\rightarrow \nabla V\mbox{ strongly in }C([0,T);L^{2}(\T)), 
\end{equation}
and the Poisson equation is satisfied pointwise.

Regarding the momentum equation, let $\psi\in C^{\infty}([0,T);C^{\infty}(\T))$ and consider the weak formulation of the momentum equations in Definition \ref{WSQNS}, multiplied by $\e$, 
\begin{equation}\label{eq:p1}
\begin{aligned}
&\e\int\re^0\ue^{0}\psi(0)\,dx +\e^2\iint\rre\rre \frac{\ue}{\e}\psi_t \;dsdx +\e^2\iint\rre \frac{\ue}{\e} \otimes\rre\frac{\ue}{\e}:\nabla\psi \;dsdx \\
&-\sqrt{\epsilon}\iint\rre\frac{\MT_{\e}^{s}}{\sqrt{\epsilon}}:\nabla\psi \;dsdx +\iint\re^{\gamma}\dive\psi \;dsdx -\iint 2\rre\,\nabla^2\rre:\nabla\psi \;dsdx\\
&- \iint \re \nabla V_\e \psi \;dsdx +\iint 2\nabla\rre\otimes\nabla\rre\,:\nabla\psi \;dsdx+\iint\re\frac{\ue}{\e}\psi \;dsdx=0.
\end{aligned}
\end{equation}
We study the convergence in the limit of $\e\to0$ of all the terms separately. 
By using \eqref{eq:idn} and H\"older inequality we conclude
\begin{equation*}
\epsilon \left|\int \rho_\epsilon^0 u_\epsilon^0 \psi(0)\,dx\right| \leq \epsilon\,\|\psi\|_{L^{\infty}_{t,x}} \|\sqrt{\re^0}\|_{L^{2}} \|\sqrt{\rho^0_{\epsilon}} u^0_\epsilon\|_{L^2} \leq \epsilon\, C \rightarrow 0 \text{ as $\epsilon \rightarrow 0 $}.
\end{equation*}
Analogously, from  \eqref{uniformbounds} and H\"older inequality, we  get  for $\epsilon \rightarrow 0 $:
\begin{equation*} 
\begin{aligned}
&\epsilon^2 \left|\iint \sqrt{\rho_\epsilon} \sqrt{\rho_\epsilon} \frac{u_\epsilon}{\e} \psi_t \;dsdx \right| \leq \epsilon^2 \|\psi_t\|_{L^{\infty}_{t,x}}\|\sqrt{\rho_{\epsilon}}\|_{L^2_{t,x}}\left\|\sqrt{\rho_\epsilon}\,\frac{u_\epsilon}{\e}\right\|_{L^2_{t,x}} \leq \epsilon^2 C \rightarrow 0,\\
&\epsilon^2 \left|\iint\sqrt{\rho_\epsilon}\,\frac{u_\epsilon}{\e} \otimes \sqrt{\rho_\epsilon}\, \frac{u_\epsilon}{\e}:  \nabla \psi \;dsdx \right| \leq \epsilon^2 \|\nabla \psi\|_{L^{\infty}_{t,x}} \left\|\sqrt{\rho_\epsilon}\,\frac{u_\epsilon}{\e}\right\|^2_{L^2_{t,x}} \leq \epsilon^2 C \rightarrow 0, \\
&\sqrt{\epsilon}\left|\iint \sqrt{\rho_\epsilon} \frac{\MT_{\e}^{s}}{\sqrt{\epsilon}} : \nabla \psi \;dsdx \right| \leq \sqrt{\epsilon}\|\nabla \psi\|_{L^{\infty}_{t,x}}\|\sqrt{\rho_{\epsilon}}\|_{L^2_{t,x}}\left\|\frac{\MT_{\e}^{s}}{\sqrt{\epsilon}}\right\|_{L^2_{t,x}} \leq \sqrt{\epsilon}C  \rightarrow 0.\\
\end{aligned}
\end{equation*}
Next, by using 
\eqref{eq:rho2} and
\eqref{eq:rho5} 
it follows that for $\epsilon \rightarrow 0 $
\begin{align*}
&\iint 2\sqrt{\rho_\epsilon} \nabla^2 \sqrt{\rho_\epsilon} : \nabla \psi\;dsdx \rightarrow\iint2\sqrt{\rho} \nabla^2 \sqrt{\rho} : \nabla \psi \;dsdx,  
\\
&
\iint 2\nabla \sqrt{\rho_\epsilon} \otimes \nabla\sqrt{\rho_\epsilon} : \nabla \psi \;dsdx \rightarrow \iint2\nabla \sqrt{\rho} \otimes \nabla\sqrt{\rho} : \nabla \psi \;dsdx.
\end{align*}
Moreover, the convergence of $\rho_{\epsilon}^{\gamma}$ in \eqref{eq:rho4} implies the continuity of the pressure term
\begin{equation*}
\iint \rho_{\epsilon}^{\gamma} \dive \psi \;dsdx \rightarrow \iint \rho^{\gamma} \dive \psi \;dsdx
\end{equation*}
as $\epsilon \rightarrow 0 $.
Next, we consider the potential term, by using the \eqref{eq:rho1} and \eqref{eq:convv} one gets:
\begin{equation*}
    \iint \re \nabla V_\e \psi \;dsdx \rightarrow \iint \rho \nabla V \psi \;dsdx \text{ as } \e \rightarrow 0.
    \end{equation*}
For  the damping term, we first note that, that from \eqref{uniformbounds}, we can infer that there exists $\Lambda$ such that 
\begin{equation}
\rre\,\frac{\ue}{\e}\weakto\Lambda\text{ weakly in $L^2((0,T)\times\T)$},\label{eq:rho7}
\end{equation}
 by using also \eqref{eq:rho5} we get that 
\begin{equation*}
\iint  \rre \rre\,\frac{u _\epsilon}{\e} \psi \;dsdx \rightarrow \iint \sqrt{\rho} \Lambda \psi \;dsdx \,\text{ as }\epsilon \rightarrow 0 
\end{equation*} 
Therefore from \eqref{eq:p1} we conclude
\begin{equation*}
\begin{split}
\iint \sqrt{\rho} \Lambda \psi \;dsdx = & -2\iint \nabla \sqrt{\rho} \otimes \sqrt{\rho} : \nabla \psi \;dsdx + 2\iint \sqrt{\rho} \nabla^2 \sqrt{\rho} : \nabla \psi \;dsdx \\
&- \iint \rho^{\gamma} \dive \psi \;dsdx + \iint \rho \nabla V \psi \;dsdx,
\end{split}
\end{equation*} 
that is
\begin{equation}\label{eq:lambda}
\sqrt{\rho} \Lambda = \dive (-2 \nabla \sqrt{\rho} \otimes \sqrt{\rho} + 2 \sqrt{\rho}  \nabla^2 \sqrt{\rho} - \rho^{\gamma} \mathbb{I}) + \rho \nabla V \quad \text{in } \mathcal{D}'((0,T)\times\T).
\end{equation}
Finally, for the continuity equation we similarly have for $\e\to 0$:
\begin{equation*}
\int \rho_\epsilon^0 \phi(0)\,dx + \iint \rho_\epsilon \phi_t + \sqrt{\rho_\epsilon} \sqrt{\rho_\epsilon} u_\epsilon \nabla \phi \;dsdx \to \int\rho^0 \phi(0) \; dx + \iint \rho \phi_t + \sqrt{\rho} \Lambda \nabla \phi \;dsdx,
\end{equation*}
and therefore taking into account \eqref{eq:lambda} we get that $\rho$ satisfies 
\begin{equation}\label{eq:wfbis}
\iint \rho \phi_t + \sqrt{\rho} \Lambda \nabla \phi \;dsdx =0,
\end{equation}
or any $\phi\in C^{\infty}_{c}([0,T);C^{\infty}(\T))$. 
Next, we prove \eqref{eq:idens}. Again, from \eqref{uniformbounds} we have that there exists $\mathcal{S}$ such that 
\begin{equation}
\mathcal{S}_{\e}\weakto\mathcal{S}\text{ weakly in $L^2((0,T)\times\T)$},\label{eq:rho8}
\end{equation}
Therefore, for any $\phi\in C^{\infty}(\T)$, we have that 
\begin{equation*}
\begin{aligned}
&\iint\rre\mathcal{S}_{\e}\phi \;dsdx \to \iint\rrho\mathcal{S}\phi \;dsdx \\
&\iint \rre\nabla^2\rre\phi-\nabla\rre\otimes\nabla\rre \;dsdx \to \iint \rrho\nabla^2\rrho\phi-\nabla\rrho\otimes\nabla\rrho \;dsdx
\end{aligned}
\end{equation*}
where we have used \eqref{eq:rho5}, \eqref{eq:rho2} and \eqref{eq:rho8}. Finally, by using \eqref{eq:dispdiss} we get \eqref{eq:idens}. 
Next, we prove the entropy inequalities. By lower semicontinuity we have that for a.e. $t\in(0,T)$
\begin{equation}\label{eq:ene}
\begin{aligned}
&\int_{\T} \left(\frac{\rho(t,x)^\gamma}{\gamma - 1}+2 |\nabla \sqrt{\rho}(t,x)|^2 + \frac{1}{2}|\nabla V(t,x)|^2 \right)\,dx+ \int_0^{t}\int_{\T} |\Lambda(t,x)|^2\;dsdx  \\
	& \leq\liminf_{\e\to 0} \int_{\T} \left(  \frac{1}{2} \re^0|\ue^0|^2 + \frac{({\re^{0})}^{\gamma}}{\gamma - 1} +  2 |\nabla \sqrt{\re^0}|^2 + \frac{1}{2}|\nabla V_\e^0|^2 \right)\;dx, 
\end{aligned}
\end{equation}
and then \eqref{eq:form11} and \eqref{eq:form23} follows by using \eqref{eq:idn} and \eqref{eq:convid}, respectively, and \eqref{eq:convv}.

Finally, regarding \eqref{eq:form22}, we recall that we only assume \eqref{eq:idn}. We first note that \eqref{eq:idn} implies that, up to a subsequence, 
\begin{equation*}
\re^0\to\rho^0\,\text{ a.e. in }(0,T)\times\T
\end{equation*}
then by using \eqref{eq:log1} and the very same argument used in Lemma \ref{lem:rho} to prove \eqref{eq:rhologrho} we get that 
\begin{equation*}
\begin{aligned}
&\re^0(\log\re^0+1)+1\to \rho^0(\log\rho^0+1)+1\text{ strongly in }L^{1}(\T). 
\end{aligned}
\end{equation*}
Moreover, we have that
\begin{equation*}
    \lim_{\e \rightarrow 0} \int_0^t \int_{\T} \re(\re-g_\e) \;dsdx = \lim_{\e \rightarrow 0} \int_0^t \int_{\T} \re^2 \;dsdx - \lim_{\e \rightarrow 0} \int_0^t \int_{\T} \re g_\e \;dsdx = \int_0^t \int_{\T} \rho(\rho-g) \;dsdx,
\end{equation*}
but this follows directly from \eqref{eq:rho1} and the weak convergence of $g_\e$. 
Therefore considering \eqref{newbdestimate1} and arguing exactly as done to deduce \eqref{eq:ene} we get \eqref{eq:form22}.

\section*{Appendix A}
For completeness, we give the proof of \eqref{nablalogrho}. 
\begin{thm}
There exists $C>0$, depending only on the dimension, such that for any function $\rho\in H^{2}(\T)$, with $\rho>0$ a.e. on $(0,T)\times\T$
 		\begin{equation*}
		\iint|\nabla\rho^{\frac{1}{4}}|^4+\iint|\nabla^2\rrho|^2\leq C \iint\rho|\nabla^2\log\rho|^2.
		\end{equation*}
		\end{thm}
		\begin{proof}
By a density argument it is enough to prove the lemma for $\rho$ being a smooth function strictly positive everywhere. We first notice that
	\begin{align}\label{1}
	\iint \rho \left|\nabla\left(\frac{\nabla \sqrt{\rho}}{\sqrt{\rho}} \right)\right|^2 = \iint \rho\left|- \frac{1}{2}\nabla \log \rho \otimes \nabla \log \rho +  \frac{1}{2\rho} \nabla^2 \rho \right|^2 & = \frac{1}{4}\iint \rho |\nabla^2\log \rho|^2.
	\end{align}
	On the other hand we also have
	\begin{align}\label{2}
	& \iint \rho \left|\nabla\left(\frac{\nabla \sqrt{\rho}}{\sqrt{\rho}} \right)\right|^2 = \iint \frac{1}{\rho}|\nabla \sqrt{\rho}|^4 + |\nabla^2 \sqrt{\rho}|^2 - 2 \frac{1}{\sqrt{\rho}} \nabla^2 \sqrt{\rho} :  \nabla \sqrt{\rho} \otimes \nabla \sqrt{\rho}.
	\end{align}
	We have:
	\begin{align*}
	& \iint \frac{1}{\sqrt{\rho}} \partial_{x_i}\partial_{x_j} \sqrt{\rho}   \partial_{x_i} \sqrt{\rho}  \partial_{x_j} \sqrt{\rho} = \iint \partial_{x_i}\left(\partial_{x_j}\sqrt{\rho} \frac{\partial_{x_i}\sqrt{\rho}}{\sqrt{\rho}}  \partial_{x_j}\sqrt{\rho}\right)  
	\\
	&- \iint \partial_{x_j}\sqrt{\rho} \partial_{x_i} \left( \frac{\partial_{x_i}\sqrt{\rho}}{\sqrt{\rho}} \right) \partial_{x_j}\sqrt{\rho} - \iint \frac{1}{\sqrt{\rho}} \partial_{x_i}\partial_{x_j} \sqrt{\rho} \partial_{x_i} \sqrt{\rho}  \partial_{x_j} \sqrt{\rho}.
	\end{align*}
	The first term is zero and thus we get
	\begin{align*}
	2 \iint \frac{1}{\sqrt{\rho}} \nabla^2 \sqrt{\rho} : \nabla \sqrt{\rho}\otimes \nabla \sqrt{\rho}  = - \iint |\nabla \sqrt{\rho}|^2 \dive \left( \frac{\nabla\sqrt{\rho}}{\sqrt{\rho}} \right). 
	\end{align*}
	Then, we use Young inequality
	\begin{align*}
	2 \left |\iint \frac{1}{\sqrt{\rho}} \nabla^2 \sqrt{\rho} :  \nabla \sqrt{\rho} \otimes \nabla \sqrt{\rho} \right | \leq 	  \frac{1}{2}\iint \frac{|\nabla \sqrt{\rho}|^4}{\rho} +   2\iint \rho  
	\left |\dive \left( \frac{\nabla\sqrt{\rho}}{\sqrt{\rho}} \right)\right |^2
	\\
	\leq  \frac{1}{2} \iint \frac{|\nabla \sqrt{\rho}|^4}{\rho} +  2 \iint 
	\rho \left|\nabla\left(\frac{\nabla \sqrt{\rho}}{\sqrt{\rho}} \right)\right|^2,
	\end{align*}
	and finally we get, using \eqref{1} and \eqref{2}:
	\begin{align*}
	\iint \frac{|\nabla \sqrt{\rho}|^4}{\rho} + |\nabla^2 \sqrt{\rho}|^2 \leq C \iint \rho |\nabla^2 \log \rho|^2
	\end{align*}
	that gives \eqref{nablalogrho}, being 
	\begin{equation*}
	\iint |\nabla \rho^{\frac{1}{4}}|^4 =  \iint \frac{|\nabla \sqrt{\rho}|^4}{\rho}.
	\end{equation*}
\end{proof}


\begin{thebibliography}{99}

\bibitem{AM} Antonelli P., Marcati P.,  On the finite energy weak solutions to a system in quantum fluid dynamics, {\em  Comm. Math. Phys.} 287 (2009), 657--686.



\bibitem{AM_CM} Antonelli P.,  Marcati P., Some results on systems for quantum fluids, \emph{ Recent Advances in Partial Differential Equations and Application}, Cont. Math. 666 (2016), 41--54.

\bibitem{AS} Antonelli P., Spirito S.,  On the compactness of finite energy weak solutions to the quantum Navier-Stokes equations, \emph{J. Hyperbolic Differ. Equ.}, 15 (2018), 133-147.



\bibitem{AS1} Antonelli P.,  Spirito S., Global existence of finite energy weak solutions of quantum Navier-Stokes equations, \emph{ Arch. Ration. Mech. Anal.} 225 (2017), 1161--1199.


\bibitem{AS3} Antonelli P., Spirito S., On the compactness of weak solutions to the Navier-Stokes-Korteweg equations for capillary fluids, \emph{Nonlinear Anal.} 187 (2019), 110--124.

\bibitem{ASnew}
Antonelli P., Spirito S., Global existence of weak solutions to the Navier-Stokes-Korteweg equations, preprint \url{arXiv:1903.02441}, (2019).


\bibitem{BW}  Baccarani G., Wordeman M.R., An investigation of steady-state velocity overshoot in silicon, \emph{Solid-State Elec.}, 28 (1985), 407--416.

\bibitem{BF} Bernis F., Friedman A., Higher-order nonlinear degenerate parabolic equations, \emph{J. Differential Equations},  83, (1990), 179--206.

\bibitem{BlubCR} Bresch D., Colin M.,  Msheik K.,  Noble P.,  Song X., BD entropy and Bernis-Friedman entropy,
\emph{ C. R. Math. Acad. Sci. Paris} 357 (2019),1--6.


\bibitem{Blub} Bresch D., Colin M.,  Msheik K.,  Noble P.,  Song X., Lubrication and shallow-water systems. Bernis-Friedman and BD entropies, available online.

%\bibitem{BD} Bresch D.,  Desjardins D.,  Existence of Global Weak Solutions for a 2D Viscous Shallow Water Equations and Convergence to the Quasi-Geostrophic Model, \emph{ Comm. Math. Phys.},  238 (2003), 211--223.




%\bibitem{BDL} Bresch D., Desjardins B.,  Lin C.K., On some compressible fluid models: Korteweg, lubrication, and shallow water systems, \emph{Comm. Partial Differential Equations}, 28 (2003), 843--868.















 %\bibitem{BGLV} Bresch D., Gisclon M., Lacroix-Violet I., On Navier–Stokes–Korteweg and Euler–Korteweg Systems: Application to Quantum Fluids Models, \emph{Arch. Rational Mech. Anal.}, 233 (2019),975--1025.
 

\bibitem{BGLVV} Bresch D., Gisclon M., Lacroix-Violet I.,  Vasseur A., On the exponential decay for compressible Navier-Stokes-Korteweg equations with a drag term, preprint \url{arXiv:2004.07895}, (2020).

\bibitem{BVY} Bresch D., Vasseur A.,  Yu C., Global Existence of Entropy-Weak Solutions to the Compressible Navier-Stokes Equations with Non-Linear Density Dependent Viscosities,  preprint \url{arXiv:1905.02701}, (2019). 

\bibitem{BM} Brull S.,  M\'ehats F., Derivation of viscous correction terms for the isothermal quantum Euler model, \emph{ZAMM Z. Angew. Math. Mech.}, 90 (2010), 219--230.



\bibitem{CCL} Cianfarani Carnevale G., Lattanzio C., High friction limit for Euler-Korteweg and Navier-Stokes-Korteweg models via relative entropy approach, \emph{ J. Differential Equations},  269 (2020), 10495--10526.


\bibitem{DMR05} Degond P., M\'ehats F., C. Ringhofer, Quantum Energy-Transport and Drift-Diffusion Models, \emph{J. Stat. Phys.}, 118 (2005), 625--667.  


\bibitem{DR03} Degond P., Ringhofer C., Quantum Moment Hydrodynamics and the Entropy Principle, \emph{ J. Statist. Phys.},112 (2003), 587--628.


\bibitem{DMTrans}
Donatelli D., Marcati P., Convergence of singular limits for multi-D semilinear hyperbolic systems to parabolic systems, \emph{ Trans. Amer. Math. Soc.}, 356 (2004), 2093--2121.


%\bibitem{DS} Dunn J.E.,  Serrin J., On the thermomechanics of interstitial working, \emph{Arch. Rational Mech. Anal.}, 88 (1985), 95--133.

%\bibitem{Gar} Gardner C., The quantum hydrodynamic model for semincoductor devices, \emph{SIAM J. Appl. Math.}, 54 (1994), 409--427. 


\bibitem{GST} Gianazza U., G. Savar\'e G., Toscani G., The Wasserstein gradient flow of the Fisher information and the quantum drift-diffusion equation, \emph{ Arch. Ration. Mech. Anal.},  194 (2009), 133--220.


\bibitem{GLT17}
Giesselmann J., Lattanzio C., Tzavaras A.E., Relative energy for the Korteweg theory and related Hamiltonian flows in gas dynamics, \emph{ Arch. Ration. Mech. Anal.}, 223 (2017), 1427--1484.

\bibitem{GJT} Gualdani M.P., J\"ungel A., Toscani G., Exponential decay in time of solutions of the viscous quantum hydrodynamic equations, \emph{ Appl. Math. Lett.}, 16 (2003), 1273--1278.

\bibitem{HLM} Huang F.,  Li H.L., Matsumura A., Existence and stability of steady-state of one-dimensional quantum hydrodynamic system for semiconductors, \emph{ J. Differential Equations},  225 (2006), 1--25.



\bibitem{HLMO} Huang F.,  Li H.L., Matsumura A., Odanaka S., Well-posedness and stability of quantum hydrodynamics for semiconductors in $\mathbb R^3$, \emph{ Some problems on nonlinear hyperbolic equations and applications}, 131--160, Ser. Contemp. Appl. Math. CAM, 15, \emph{ Higher Ed. Press, Beijing}, 2010.




\bibitem{Jun_transp} J\"ungel A., Transport equations for semiconductors,  Lecture Notes in Physics, 773. \emph{Springer-Verlag, Berlin}, 2009.





%\bibitem{Jun_SIAM} J\"ungel A.,  Global weak solutions to compressible Navier-Stokes equations for quantum fluids, \emph{ SIAM J. Math. Anal.},  42 (2010), 1025--1045.





\bibitem{Jun_rev} J\"ungel A.,  Dissipative quantum fluid models.,
\emph{Riv. Math. Univ. Parma (N.S.)},  3 (2012), 217–290.



\bibitem{JLM_rel} J\"ungel A., Li H.L., Matsumura A., The relaxation-time limit in the quantum hydrodynamic equations for semiconductors, \emph{J. Differential Equations},  225 (2006),  440--464.



\bibitem{JLM_WFP}  J\"ungel A.,  L\'opez J.L., Montejo-G\'amez J.,  A new derivation of the quantum Navier-Stokes equations in the Wigner-Fokker-Planck approach, \emph{ J. Stat. Phys.},  145 (2011), 1661--1673.

\bibitem{JM} J\"ungel A., Matthes D., The Derrida-Lebowitz-Speer-Spohn equation: existence, non-uniqueness,
and decay rates of the solutions, \emph{SIAM J. Math. Anal.}, 39 (2008), 1996--2015.

\bibitem{JM_fQNS} J\"ungel A.,  Mili\v{s}i\'c J.C., Full compressible Navier-Stokes equations for quantum fluids: derivation and numerical solution, \emph{Kinet. Relat. Models }, 4 (2011), 785--807.


%\bibitem{Khal} Khalatnikov I., \emph{An Introduction to the Theory of Superfluidity}, 2000.

%\bibitem{Kor} Korteweg D.J., Sur la forme que prennent les \'equations du mouvement des fluides si l'on tient en compte des forces capillaires caus\'ees par des variations de densit\'e, \emph{Archives N\'eerl. Sci. Exactes Nat. Ser. II}, 60 (1901), 1--24.

\bibitem{LZ} Liang B., Zheng S., Exponential decay to a quantum hydrodynamic model for semiconductors, \emph{Nonlinear Anal. Real World Appl. }, 9 (2008), 326--337.


%\bibitem{LZZ} L\"u B., Zhang R., Zhong X.,Global existence of weak solutions to the compressible quantum Navier-Stokes equations with degenerate viscosity, archived as \url{ arXiv:1906.0397}.

\bibitem{LV} Lacroix-Violet I., Vasseur A.,  Global weak solutions to the compressible quantum Navier-Stokes equation and its semi-classical limit, \emph{J. Math. Pures Appl.}, 114 (2017), 191--210. 

\bibitem{Latt} Lattanzio C., On the $3-D$ bipolar isentropic Euler-Poisson model for semiconductors and the drift-diffusion limit., \emph{Math. Models Methods Appl. Sci.},  10 (2000), 351--360.

\bibitem{LM_rel} Lattanzio C., Marcati P., The relaxation to the drift-diffusion system for the $3-D$ isentropic Euler-Poisson model form semiconductors, \emph{ Discrete Contin. Dynam. Systems}, 5 (1999), 449--455.



\bibitem{LT13} Lattanzio C., Tzavaras A.E., Relative entropy in diffusive relaxation, \emph{SIAM J. Math. Anal. } 45 (2013), 1563--1584.



\bibitem{LT17}  Lattanzio C., Tzavaras A.E., From gas dynamics with large friction to gradient flows describing diffusion theories, \emph{Comm. Partial Differential Equations},  42 (2017), 261--290.


%\bibitem{LX} Li J., Xin Z.,  Global Existence of Weak Solutions to the Barotropic Compressible Navier-Stokes Flows with Degenerate Viscosities, preprint \url{arXiv:1504.06826}.

\bibitem{MN} Marcati P, Natalini R., Weak solutions to a hydrodynamic model for semiconductors and relaxation to the drift-diffusion equation, \emph{ Arch. Rational Mech. Anal.},  129 (1995), 129--145.



\bibitem{MRS}
Markowich P.A., Ringhofer C., Schmeiser C., Semiconductors Equations. \emph{Springer-Verlag, Wien, New York}, 1990.

%\bibitem{PS} Pitaevskii L., Stringari S., Bose-Einstein condensation and superfluidity, \emph{Clarendon Press, Oxford}, 2016.

 \bibitem{VanR}
Roosbroeck W.V., Theory of flow of electrons and holes in germanium and other semiconductors, \emph{Bell. Syst. Techn. J.}, 29 (1950), 560--607.

\bibliographystyle{plain}
\end{thebibliography}
\end{document}